\documentclass[11pt,a4paper]{amsart}
\usepackage{amssymb,xspace}
\usepackage{amstext}
\theoremstyle{plain}
\usepackage{amsbsy,amssymb,amsfonts,latexsym}

\usepackage{enumerate}
\usepackage{graphics}

\date{\today}
\title{Central limit theorems in linear dynamics}
\author{Fr\'ed\'eric Bayart}
\address{
Clermont Universit\'e, Universit\'e Blaise Pascal, Laboratoire de Math\'ematiques, BP 10448, F-63000 CLERMONT-FERRAND -
CNRS, UMR 6620, Laboratoire de Math\'ematiques, F-63177 AUBIERE
}
\email{Frederic.Bayart@math.univ-bpclermont.fr}

\subjclass{47B37}

\keywords{hypercyclic operators, linear dynamics, ergodic theory, dynamical systems, central limit theorem}

\newcommand{\dis}{\displaystyle}
\newcommand{\veps}{\varepsilon}

\def\RR{\mathbb R}

\def\NN{\mathbb N}
\def\ZZ{\mathbb Z}
\def\TT{\mathbb T}
\def\DD{\mathbb D}

\def\CC{\mathbb C}

\def\LX{\mathfrak{L}(X)}

\def\dmunk{d\bar \mu\big((n_k)\big)}

\def\NNinf{\NN^\infty}
\def\ZZinf{\ZZ^\infty}
\def\ZZinfinf{\ZZ_<^\infty}

\newcommand{\vect}{\textrm{span}}
\newcommand{\cov}{\mathrm{Cov}}

\newtheorem{theorem}{Theorem}[section]

\newtheorem{lemma}[theorem]{Lemma}

\newtheorem{proposition}[theorem]{Proposition}

\newtheorem{corollary}[theorem]{Corollary}

{\theoremstyle{definition}}
{\theoremstyle{definition}}

{\theoremstyle{definition}\newtheorem{example}[theorem]{Example}}

{\theoremstyle{definition}}

{\theoremstyle{definition}}

{\theoremstyle{definition}\newtheorem{remark}[theorem]{Remark}}

\newtheorem{question}[theorem]{Question}

\newtheorem*{theoremhcc}{Theorem (Hypercyclicity Criterion)}
\newtheorem*{VOLNY}{Theorem A}

\newtheorem*{Fact1}{Fact 1}
\newtheorem*{Fact2}{Fact 2}

\begin{document}

\begin{abstract}
Given a bounded operator $T$ on a Banach space $X$, we study the existence of a probability measure
$\mu$ on $X$ such that, for many functions $f:X\to\mathbb K$, the sequence $(f+\dots+f\circ T^{n-1})/\sqrt n$
converges in distribution to a Gaussian random variable.
\end{abstract}

\maketitle

\section{Introduction}

Linear dynamics (namely the study of the dynamics of linear operators) is a branch of analysis
connecting functional analysis and dynamics. Its main topics are detailed in the two books \cite{BM09} and \cite{GePeBook}. As for the classical dynamical systems, one can study the dynamics of linear operators from a topological point of view.  Precisely, an operator $T$ defined on a separable Banach space $X$ is called \emph{hypercyclic} provided there exists a vector $x\in X$
such that its orbit $\{T^n x;\ n\geq 0\}$ under $T$ is dense. In this context, contrary to the general case, there is a very easy criterion to prove that an operator is hypercyclic; it allows
to exhibit many hypercyclic operators. Let us recall this criterion.
\begin{theoremhcc}
Let $T\in\LX$.
Suppose that there exist a dense subset $\mathcal D\subset X$ and a sequence of maps $(S_n)_{n\geq 0}$, $S_n:\mathcal D\to X$, such that, for each $x\in\mathcal D$,
\begin{enumerate}[(i)]
\item $T^nx\to 0$;
\item $S_n x\to 0$;
\item $T^n S_n x\to x$.
\end{enumerate}
Then $T$ is hypercyclic.
\end{theoremhcc}

In this paper, we shall concentrate on the other aspect of linear dynamics, that links it with ergodic theory. Let us recall some basic definitions. Let $(X,\mathcal B,\mu)$ be a probability space and let $T:(X,\mathcal B,\mu)\to (X,\mathcal B,\mu)$ be a measurable map. We set 
$$L^2_0(\mu)=\left\{f\in L^2(\mu);\ \int_X fd\mu=0\right\}.$$

We say that $T$ is a \emph{measure-preserving transformation} (or that $\mu$ is $T$-invariant)
if $\mu(T^{-1}A)=\mu(A)$ for any $A\in\mathcal B$. A measure-preserving transformation $T:(X,\mathfrak B,\mu)\to(X,\mathfrak B,\mu)$ is \emph{ergodic} (with respect to $\mu$) if
$$\frac{1}{N}\sum_{n=0}^{N-1} \mu(A\cap T^{-n}(B))\xrightarrow{N\to\infty} \mu (A)\mu(B)
$$
for any measurable sets $A,B\subset X$; $T$ is \emph{weakly mixing} (with respect to $\mu$) if
$$\frac{1}{N}\sum_{n=0}^{N-1} \vert \mu(A\cap T^{-n}(B))-\mu (A)\mu(B)\vert\xrightarrow{N\to\infty} 0
$$for any $A,B\in\mathcal B$; and $T$ is \emph{strongly mixing} if
$$\mu(A\cap T^{-n}(B))\xrightarrow{n\to\infty} \mu(A)\mu(B)$$
for any $A,B\in\mathcal B$. 

We wonder whether there exists a Borel probability  measure $\mu$ on the separable Banach space $X$, which is nondegenerate (namely, $\mu(U)>0$ for any nonempty and open subset $U\subset X$),
such that $T\in\mathcal L(X)$ is $\mu$-invariant and $T$ is ergodic (resp. weakly mixing, strongly mixing) with respect to $\mu$.  This line of investigation was opened by Flytzanis in \cite{Fly95} and pursued later by Bayart, Grivaux, Matheron (see \cite{BAYGRITAMS}, \cite{BAYGRIPLMS}, \cite{BAYMATHERGOBEST})
 and also recently by Murillo-Arcilla and Peris (see \cite{MuPe13}). 

It turns out that, when an operator has many eigenvectors associated to eigenvalues of modulus 1, then it is weakly mixing with respect to a nondegenerate and  invariant Gaussian measure on $X$. Indeed, the following theorem is proved in \cite{BAYMATHERGOBEST}.
\begin{theorem}\label{THMERGOBEST}
Let $T\in\LX$. Suppose that, for any countable set $D\subset\mathbb T=\{z\in\CC;\ |z|=1\}$, 
the linear span of $\bigcup_{\lambda\in\mathbb T\backslash D}\ker(T-\lambda)$ is dense in $X$. 
Then there exists a $T$-invariant weakly  mixing Gaussian Borel probability measure
$\mu$ on $X$ with full support.
\end{theorem} 
An interesting application of this theorem (strictly speaking, of a variant of this theorem) is that an enhancement of the hypercyclicity criterion leads to a strongly mixing dynamical system.
\begin{theorem}\label{THMMURILLOPERIS}
Let $T\in\LX$. Suppose that there exist a dense set $\mathcal D\subset X$ and a sequence 
of maps $(S_n)_{n\geq 0}$, $S_n:\mathcal D\to X$, such that, for each $x\in\mathcal D$,
\begin{enumerate}[(i)]
\item $\sum_{n\geq 0} T^nx$ converges unconditionally;
\item $\sum_{n\geq 0} S_n x$ converges unconditionally;
\item $T^n S_n x= x$ and $T^mS_n x=S_{n-m}x$ if $n>m$.
\end{enumerate}
Then there exists a $T$-invariant strongly mixing (Gaussian) Borel probability measure
$\mu$ on $X$ with full support.
\end{theorem}
Theorem \ref{THMMURILLOPERIS} has also been obtained in \cite{MuPe13} in a completely different way. The measure $\mu$ constructed in \cite{MuPe13} is not a Gaussian measure; in \cite{MuPe13}, very few properties of this measure are proved. For instance, it is not known whether the norm $\|\cdot\|$ or the linear functionals $\langle x^*,\cdot\rangle$ belong to $L^2(X,\mu)$, even if we can always ensure these properties, see Section \ref{SECPROOF}. On the other hand, 
the dynamical system $(T,\mu)$ is conjugated to an easy strongly mixing dynamical system:  a Bernoulli shift. 

\medskip

In this paper, we are interested in properties stronger than just ergodicity or mixing. First we are interested in the speed of mixing. 
For $f,g\in L^2(X,\mu)$, define the correlation of $f$ and $g$ as 
$$\cov(f,g)=\int_X fgd\mu-\int_X fd\mu \int_X gd\mu$$
and the correlation of order $n$ of $f$ and $g$ (with respect to $T$) by
$$\mathcal I_n(f,g)=\cov(f\circ T^n,g).$$
$T$ is $\mu$-mixing provided $\mathcal I_n(f,g)$ goes to zero for any $f,g\in L^2(X,\mu)$. One may ask if we can estimate the speed of convergence to 0
of $\mathcal I_n(f,g)$ for any $f,g\in L^2(X,\mu)$, or at least for a large class of functions.

The second direction we investigate is related to the central limit theorem. Indeed, if $T$ is $\mu$-ergodic, 
then Birkhoff's theorem says that, for any $f\in L^1(X,\mu)$, 
$$\frac{f+\dots+f\circ T^{n-1}}{n}\to\int fd\mu\textrm{ almost surely}.$$
In other words, the sequence $(f\circ T^k)$ satisfies the strong law of large numbers. 
One can ask whether is also satisfies the central limit theorem, namely
if the sequence $\left(\frac{f+\dots+f\circ T^{n-1}}{\sqrt n}\right)$ converges in
distribution to a Gaussian random variable of zero mean and finite variance (we now assume
$f\in L^2_0$). 

At this point, it is important to notice that we cannot expect results true for all $f\in L^2_0(\mu)$. Indeed, D. Voln\'y has proved in \cite{Vol90} that the Birkhoff means may converge
to arbitrary laws on a dense $G_\delta$-set of $L^2_0(\mu)$. More precisely, Voln\'y has shown that, if $T$ is $\mu$-ergodic with $\mu(\{x\})=0$ for any $x\in X$, 
there exists a dense $G_\delta$-set $E$ of $L^2_0(\mu)$ such that, for any $f\in E$ and any probability measure $\nu$ on $\mathbb R$ satisfying $\int td\nu(t)=0$
and $\int t^2 d\nu(t)=1$, there exists a sequence $(N_k)$ tending to infinity such that $S_{N_k}/\|S_{N_k}\|_2$ converges in distribution to $\nu$, where 
$S_N=f+\dots+f\circ T^{N-1}$. Similarly, the slow decay of correlations is a typical feature of functions in $L^2_0(\mu)$ (see \cite{Che95}).
However, for many concrete dynamical systems, positive results were obtained if we assume some regularity condition on $f$, like $f$ is H\"older.

\medskip

Thus, in this paper, we are interested to prove central limit theorems or to estimate the decay of correlation for functions belonging to a large class of $L_0^2(\mu)$, in the context
of linear dynamical systems. The first step in that direction was done by V. Devinck in \cite{Dev13}. He started from Theorem \ref{THMERGOBEST} and he  was able to prove that,
when the $\TT$-eigenvectors of $T$ can be parametrized in a regular way, then $\cov(f\circ T^n,g)$ decreases fast to zero for $f,g$ belonging to large classes of functions.
Let us summarized his main result.

\begin{theorem}\label{THMDEVINCK}
Let $H$ be a separable Hilbert space and let $T\in\mathfrak L(H)$. Suppose that there exists $\alpha\in(0,1]$ and $E:\TT\to X$ such that 
$TE(\lambda)=\lambda E(\lambda)$ for any $\lambda\in\TT$ and $E$ is $\alpha$-H\"olderian: there exists $C_E>0$ such that
$$\|E(e^{i\theta})-E(e^{i\theta'})\|\leq C_E|\theta-\theta'|^\alpha\textrm{ for any }\theta,\theta'\in [0,2\pi).$$
Suppose moreover that $\vect\big(E(\lambda);\ \lambda\in\TT\big)$ is dense in $H$. Then there exist a $T$-invariant ergodic Gaussian measure $\mu$ with full support and two classes of functions $\mathcal X,\mathcal Y$ such that, for any $(f,g)\in\mathcal X\times\mathcal Y$, 
$$|\cov(f\circ T^n,g)|\leq C_{f,g}n^{-\alpha}.$$
\end{theorem}

We will not describe the two classes $\mathcal X$ and $\mathcal Y$ in the above theorem. Their definitions involve the Gaussian measure $\mu$ and depend on the Hilbertian structure of $H$. We just mention that a typical function in $\mathcal X$ or $\mathcal Y$ is an infinitely differentiable function whose sequence of derivatives satisfies some growth condition. For instance, the set of polynomials $\mathcal P$
is contained in both $\mathcal X$ and $\mathcal Y$. We recall that a function $P:X\to\mathbb R$ is a homogeneous polynomial of degree $d$ 
provided there exists a bounded symmetric $d$-linear form $Q$ on $X$ such that $P(x)=Q(x,\dots,x)$. A polynomial of degree $d$ is a sum
$P=P_0+\dots+P_d$, where each $P_k$ is a homogeneous polynomial of degree $k$.

\medskip

Theorem \ref{THMDEVINCK} is really specific to the Hilbert space setting. In this paper, we obtain several results in the Banach space setting regarding
the decay of correlations and the validity of the central limit theorem for large classes of functions. We do not start from the Gaussian measure of Theorem \ref{THMERGOBEST}; we rather use the class of measures introduced in \cite{MuPe13}. Our theorems will have the following informative form:
\begin{quote}
Let $T\in\mathfrak L(X)$ satisfying a strong form of the Hypercyclicity Criterion. Then there exist a $T$-invariant strongly mixing Borel probability measure $\mu$ on $X$ and a "large" subset $E$ of $L^2(\mu)$ such that
\begin{itemize}
\item for any $f,g\in E$, the sequence of covariances $\big(\cov(f\circ T^n, g)\big)$ converges quickly to zero;
\item for any $f\in E$, $(f+\dots+f\circ T^{n-1})/\sqrt n$ converges in distribution to a Gaussian random variable.
\end{itemize}
\end{quote}
Of course, precise statements will be given in Section \ref{SECMYTCL}, after we define our large subsets $E$. \\
\medskip

\noindent {\bf Notations.} Throughout the paper, the letter $C$ will denote an absolute constant whose value may change from line to line. 
If a constant depends on some parameter $x$, then we shall denote it by $C_x$.

\section{Central limit theorems - Results and examples} \label{SECMYTCL}

\subsection{The spaces of functions}
We shall first define the spaces of functions $f$ such that the sequence $(f\circ T^n)$ will satisfy a central limit theorem. 
There are many differences with the classical cases due to the noncompactness of $X$. We will require that $f$ is infinitely differentiable; this is stronger
than in many situations where a central limit theorem for dynamical systems has been proved. On the contrary, we do not want to restrict ourselves to bounded functions. We expect to apply our results to linear forms for instance.

Let $\omega:\mathbb N\to(1,+\infty)$ which goes to infinity. We define $E_\omega$ as the set of functions $f:X\to\mathbb R$
which are infinitely differentiable at $0$, which may be written, for any $x\in X$, 
$$f(x)=\sum_{\kappa=0}^{+\infty}\frac{D^\kappa f(0)}{\kappa !}(x,\dots,x),$$
and whose sequence of derivatives satisfies
$$\|f\|_\omega:=\sup_{\kappa\geq 0}\|D^\kappa f(0)\|\omega(\kappa)^\kappa<+\infty.$$
Endowed with the norm $\|\cdot\|_\omega$, $E_\omega$ is a Banach space. Clearly, each $E_\omega$ contains the polynomials.
When $\omega$ is well chosen, it also contains other natural classes of functions, as the following proposition
indicates.
\begin{proposition}
There exists a function $\omega:\mathbb N\to\mathbb (1,+\infty)$ tending to infinity such that, for any polynomial $P\in\mathcal P$, for any function
$\phi:\mathbb R\to\mathbb R$ which can be written $\phi(x)=\sum_{n\geq 0}\frac{a_n}{(n!)^{\sigma}}x^n$ with $|a_n|\leq A\tau ^n$ for some constants 
$A,\tau>0$ and $\sigma>\deg(P)$, $\phi\circ P$ belongs to $E_\omega$.
\end{proposition}
\begin{proof}
We write $P=P_0+\dots+P_d$, where each $P_k$ is homogeneous with degree $k$. We may assume $\deg(P)>0$. Let $B>0$ be such that
$|P_k(x)|\leq B\|x\|^k$, for any $k=0,\dots,d$. We develop $\phi\circ P$ into
\begin{eqnarray*}
\phi\circ P&=&\sum_{k=0}^{+\infty}\frac{a_k}{(k!)^\sigma}\sum_{j_0+\dots+j_d=k}P_0^{j_0}\dots P_d^{j_d}\\
&=&\sum_{j_0,\dots,j_d=0}^{+\infty}P_0^{j_0}\dots P_d^{j_d}\frac{a_{j_0+\dots+j_d}}{[(j_0+\dots+j_d)!]^\sigma}\\
&=&\sum_{l=0}^{+\infty}\sum_{\substack{j_0\geq 0\\ j_1+\dots+dj_d=l}}P_0^{j_0}\dots P_d^{j_d}\frac{a_{j_0+\dots+j_d}}{[(j_0+\dots+j_d)!]^\sigma}\\
&=:&\sum_{l=0}^{+\infty}Q_l.
\end{eqnarray*}
Each $Q_l$ is a homogeneous polynomial with degree $l$. We are looking for a function $\omega$ such that the sequence $(\|Q_k\|k!\omega(k)^k)$ is bounded. We observe that
\begin{eqnarray*}
\|Q_k\|&\leq&A\sum_{\substack{j_0\geq 0\\ j_1+\dots+dj_d=k}}\frac{(B\tau)^{j_0+\dots+j_d}}{\big((j_0+\dots+j_d)!\big)^\sigma}\\
&\leq&A\sum_{j_0\geq 0}(B\tau)^{j_0}\sum_{j_1+\dots+dj_d=k}\frac{(B\tau)^{j_1+\dots+j_d}}{\big((j_0+\dots+j_d)!\big)^\sigma}\\
&\leq&A\sum_{j_0\geq 0}(B\tau)^{j_0}\sum_{j_1+\dots+dj_d=k}\frac{(B\tau)^{k}}{\left(\left(\left[\frac kd\right]+j_0\right)!\right)^\sigma}\\
\end{eqnarray*}
(we assume, expanding $B$ if necessary, that $B\tau\geq 1$). Now, the number of solutions of the equation 
$j_1+\dots+dj_d=k$, $j_i\geq 0$, is less than $k ^d$. We deduce the existence of some positive constant $\tau_1$ (depending on $d$)
such that
$$\|Q_k\|\leq A\tau_1^k\sum_{j_0\geq 0}\frac{(B\tau)^{j_0}}{\left(\left(\left[\frac kd\right]+j_0\right)!\right)^\sigma}.$$
Now, for $k$ sufficiently large and for any $j_0\geq 0$, 
$$\frac{(B\tau)^{j_0}}{\left(\left(\left[\frac kd\right]+j_0\right)!\right)^\sigma}\leq \left(\frac12\right)^{j_0}\times\frac{1}{\left(\left[\frac kd\right]!\right)^\sigma}.$$
This yields
$$\|Q_k\|k!\leq A'\frac{\tau_1^kk!}{\left(\left[\frac kd\right]!\right)^\sigma}.$$
We now set $\omega(k)={\log (k+e)}$ (which does not depend on anything!) and observe that, by Stirling's formula, the sequence
$( \|Q_k\|k!\omega(k)^k)$ is bounded.
\end{proof}

\subsection{Statement of the results and examples}
Having described the spaces of functions where we expect that a central limit theorem holds, we can now give our main statements.
We begin with an abstract result.
\begin{theorem}\label{THMMAINTCL}
Let $T\in\LX$ and let $\omega:\mathbb N\to(1,+\infty)$ going to $+\infty$. Suppose that, for any function $\omega_0:[1,+\infty)\to(1,+\infty)$ going to infinity, one can find a sequence $\mathcal D=(x_n)_{n\geq 1}\subset X$, a sequence $(\veps_k)\subset\ell^1(\ZZ)$
and a  sequence of maps $S_n:\mathcal D\to X$, $n\geq 0$, such that
\begin{enumerate}[(i)]
\item For any $x\in\mathcal D$, $\sum_{n\geq 0} T^nx$ converges unconditionally;
\item For any $x\in\mathcal D$, $\sum_{n\geq 0} S_n x$ converges unconditionally;
\item For any $x\in\mathcal D$, $T^n S_n x= x$ and $T^mS_n x=S_{n-m}x$ for any $n>m$;
\item For any sequence $(n_k)\subset\NN^{\ZZ}$ such that $\sum_{k\geq 0}T^k x_{n_k}$ and $\sum_{k<0}S_{-k} x_{n_k}$
are convergent, 
$$\left\| \sum_{k\geq 0}T^k x_{n_k}\right\|+\left\|\sum_{k<0}S_{-k} x_{n_k}\right\|\leq \sum_{k\in\ZZ}\omega_0(n_k)\veps_k.$$
\item $\textrm{span}\big(\{T^k x_n;S_k x_n\};\ k\geq 0,\ n\geq 1\big)$ is dense in $X$.
\end{enumerate}
Then there exists a $T$-invariant strongly mixing Borel probability measure $\mu$ on $X$
with full support such that $E_\omega \subset L^2(X,\mathcal B,\mu)$.\\
Suppose moreover that there exists $\alpha>1/2$ such that, for any $x\in\mathcal D$, for any $n,k\geq 0$,
$$\|T^k x_n\|\leq\frac{\omega_0(n)}{(1+k)^\alpha}\textrm{ and }\|S_k x_n\|\leq\frac{\omega_0(n)}{(1+k)^{\alpha}}.$$
Then, for any $f,g\in E_\omega$,
$$|\cov(f\circ T^n,g)|\leq C_{f,g,\alpha}
\left\{
\begin{array}{ll}
n^{1-2\alpha}&\textrm{ if }\alpha\in(1/2,1)\\
\displaystyle \frac{\log (n+1)}n&\textrm{ if }\alpha=1\\
n^{-\alpha}&\textrm{ if }\alpha>1.
\end{array}\right.
$$
If $\alpha>1$, then for any $f\in E_\omega$ with zero mean, the sequence $\dis\frac{1}{\sqrt n}(f+\dots+f\circ T^{n-1})$
converges in distribution to a Gaussian random variable of zero mean and finite variance.
\end{theorem}
 Of course, this statement does not look very appealing and we shall not prove it immediately. We prefer to give two more readable corollaries. The first one deals with stronger forms of unconditionality.

\begin{theorem}\label{THMMAINTCL2}
Let $T\in\LX$ and let $\omega:\mathbb N\to(1,+\infty)$ going to $+\infty$. Suppose that there exist $\alpha>1$, a dense set $\mathcal D\subset X$ and a  sequence of maps $S_n:\mathcal D\to X$, $n\geq 0$, such that, for any $x\in\mathcal D$,
\begin{enumerate}[(i)]
\item $T^n S_n x= x$ and $T^mS_n x=S_{n-m}x$ for any $n>m$;
\item $\|T^n x\|=O(n^{-\alpha})$ and $\|S_n x\|=O(n^{-\alpha})$.
\end{enumerate}
Then there exists a $T$-invariant strongly mixing Borel probability measure $\mu$ on $X$
with full support with $E_\omega \subset L^2(X,\mathcal B,\mu)$ and such that for any $f,g\in E_\omega$,
$$|\cov(f\circ T^n,g)|\leq C_{f,g,\alpha} n^{-\alpha}.
$$
Moreover,  for any $f\in E_\omega$ with zero mean, the sequence $\dis\frac{1}{\sqrt n}(f+\dots+f\circ T^{n-1})$
converges in distribution to a Gaussian random variable of zero mean and finite variance.
\end{theorem}

\begin{proof}
First, the assumptions imply the unconditional convergence of the series $\sum_n T^n x$ and $\sum_n S_n x$ for any $x\in\mathcal D$. Moreover, let $\omega_0:\mathbb N\to(1,+\infty)$ going to $+\infty$. Let $(x_n)_{n\geq 1}$ be a dense sequence in $\mathcal D$ such that, for any $n\geq 1$, $k\geq 0$, 
$$\|T^k x_n\|\leq\frac{\omega_0(n)}{(1+k)^\alpha}\textrm{ and }\|S_k x_n\|\leq\frac{\omega_0(n)}{(1+k)^\alpha}.$$
Then for any sequence $(n_k)\subset\NN^{\ZZ}$, 
$$\left\| \sum_{k\geq 0}T^k x_{n_k}\right\|+\left\|\sum_{k<0}S_{-k} x_{n_k}\right\|\leq \sum_{k\in\ZZ}\frac{\omega_0(n_k)}{(1+|k|)^\alpha},$$
so that the assumptions of Theorem \ref{THMMAINTCL2} are satisfied with $\veps_k=(1+|k|)^{-\alpha}$.
\end{proof}

The previous theorem may be applied to many examples where we already know that Conditions (i) and (ii) hold. For instance, adjoints of multipliers
or composition operators associated to hyperbolic automorphisms of the disk, acting on the Hardy space $H^2(\DD)$ satisfy the assumptions of Theorem \ref{THMMAINTCL2}. We refer to \cite[Chapter 6]{BM09} for a
description of these examples.

\begin{remark}
When we work on a Hilbert space, we can compare the speed of convergence to zero of the sequence of covariances
given by the measure obtained in Theorem \ref{THMMAINTCL2} with that given by Devinck's theorem.
 Indeed, the assumptions of Theorem \ref{THMMAINTCL2} imply the existence
of a sequence of perfectly spanning $\TT$-eigenvectorfields: for any $x\in\mathcal D$, put $E(\lambda)=\sum_{n\in\ZZ}\lambda^n T^n x$. If we assume that $\|T^n x\|=O(n^{-\alpha})$ and $\|S_n x\|=O(n^{-\alpha})$ for any $x\in\mathcal D$ and some
$\alpha\in(1,2)$, then one can show that $E$ is $(\alpha-1)$-H\"olderian: pick any $\lambda,\mu\in\TT$ and let $n\in\NN$ be such that
$(n+1)^{-1}<|\lambda-\mu|\leq n^{-1}$. Then
\begin{eqnarray*}
\|E(\lambda)-E(\mu)\|&\leq&C\sum_{k=1}^n k|\lambda-\mu|k^{-\alpha}+C\sum_{k>n}k^{-\alpha}\\
&\leq&C|\lambda-\mu|n^{2-\alpha}+Cn^{1-\alpha}\\
&\leq&C|\lambda-\mu|^{\alpha-1}.
\end{eqnarray*}
Thus, with the assumptions of Theorem \ref{THMMAINTCL2}, we may also apply Devinck's theorem to obtain an ergodic measure such that, 
for any $f,g\in\mathcal P$, 
$$|\cov(f\circ T^n,g)|\leq Cn^{-(\alpha-1)}.$$
This is less good than the result we obtain by applying directly Theorem \ref{THMMAINTCL2}.

Conversely, let us assume that the assumptions of Devinck's theorem are satisfied. We set 
$$\mathcal D=\vect\left(\int_{\TT}\lambda^p E(\lambda)d\lambda;\ p\in\mathbb Z\right),$$ 
which is dense in $H$. We define $S_n$ on $\mathcal D$ by
$$S_n\left( \int_{\TT}\lambda^p E(\lambda)d\lambda\right)=\int_{\TT}\lambda^{p-n} E(\lambda)d\lambda$$
so that $S_nT^m=S_{n-m}$ on $\mathcal D$, for any $n,m\geq 0$. It is not hard to show that, for any $x\in\mathcal D$,
$\|T^n x\|=O(n^{-\alpha})$ and $\|S_n x\|=O(n^{-\alpha})$. Indeed, 
\begin{eqnarray*}
T^n \left( \int_{\TT}\lambda^p E(\lambda)d\lambda\right)&=&\int_{0}^{2\pi}e^{i(n+p)\theta} E(e^{i\theta})\frac{d\theta}{2\pi}\\
&=&\frac 12\int_{0}^{2\pi}e^{i(n+p)\theta} \left(E(e^{i\theta})-E\left(e^{i\left(\theta+\frac \pi{p+n}\right)}\right)\right)\frac{d\theta}{2\pi}.
\end{eqnarray*}
by a change of variables. Since $E$ is assumed to be $\alpha$-H\"olderian, this easily implies that $\|T^n x\|=O(n^{-\alpha})$ 
for any $x\in\mathcal D$. Theorem \ref{THMMAINTCL2} and Devinck's theorem then give the same decay of correlations. 
However, we are sure that we can apply Theorem \ref{THMMAINTCL2} only if $\alpha>1$, and we cannot apply it if $\alpha<1/2$. 

Thus, in the Hilbert space
setting, it is not so easy to decide which measure has better mixing properties. Moreover, we do not know whether we can compare our class $E_\omega$ with the classes $\mathcal X$ and $\mathcal Y$ of Devinck. We also point out that the way we construct $E$ from the assumptions of Theorem \ref{THMMAINTCL2} or conversely the way we construct $\mathcal D$ from the assumptions of Devinck's theorem are often not optimal. In concrete situations, a more natural choice of $E$ (resp. of $\mathcal D$) can improve the results we get automatically in this remark. In these concrete situations, our results seem better; see the forthcoming  examples \ref{EXBS} and \ref{EXCO}.
\end{remark}

\subsection{Bilateral weighted shifts}

We now apply Theorem \ref{THMMAINTCL} to backward weighted shift operators on $\ell^p(\mathbb Z_+)$, $p\geq 1$.
Let $\mathbf w=(w_n)_{n\in\mathbb Z_+}$ be a bounded sequence in $\mathbb R_+$. The bilateral weighted shift $B_{\mathbf w}$ on 
$\ell^p(\mathbb Z_+)$ is defined by $B_{\mathbf w}\big((x_n)\big)=(w_{n+1}x_{n+1})$.
We know by \cite{BAYRUZSA} that there exists a $B_{\mathbf w}$-invariant ergodic Borel probability measure on $\ell^p(\mathbb Z_+)$ with full support
iff $\sum_{n\geq 1}(w_1\dots w_n)^{-p}<+\infty$. 
\begin{theorem}\label{THMTCLBS}
Let $\omega:\mathbb N\to(1,+\infty)$ going to $+\infty$ and let $B_{\mathbf w}$ be a bounded backward weighted shift
on $\ell^p(\mathbb Z_+)$. Suppose moreover that $\sum_{n\geq 1}(w_1\dots w_n)^{-p}<+\infty$.
Then there exists a  $B_{\mathbf w}$-invariant strongly mixing Borel probability measure on $\ell^p(\mathbb Z_+)$ with full support
such that $E_\omega\subset L^2(X,\mathcal B,\mu)$.

Suppose moreover that there exists $C>0$ and $\alpha>1/2$ such that, for any $n\geq 1$,
$$w_1\cdots w_n\geq Cn^{\alpha}.$$
Then for any $f,g\in E_\omega$, $$|\cov(f\circ B_{\mathbf w}^n,g)|\leq C_{f,g,\alpha}
\left\{
\begin{array}{ll}
n^{1-2\alpha}&\textrm{ if }\alpha\in(1/2,1)\\
\displaystyle \frac{\log (n+1)}n&\textrm{ if }\alpha=1\\
n^{-\alpha}&\textrm{ if }\alpha>1.
\end{array}\right.
$$
If $\alpha>1$, then for any $f\in E_\omega$ with zero mean, the sequence $\dis\frac{1}{\sqrt n}(f+\dots+f\circ B_{\mathbf w}^{n-1})$
converges in distribution to a Gaussian random variable of zero mean and finite variance.
\end{theorem}
\begin{proof}
Let $\omega_0:\mathbb N\to(1,+\infty)$ going to infinity. Let $(\alpha_n)_{n\geq 1}$ be a dense sequence in $\mathbb K$ with 
$|\alpha_n^p|\leq \omega_0(n)$. We set $\mathcal D=(x_n)_{n\geq 1}$, with $x_n=\alpha_n e_0$, where $(e_l)_{l\geq 0}$ is the standard basis of $\ell^p(\mathbb Z_+)$.  We define $S_n$
on $\mathcal D$ by $S_n(e_0)=\frac 1{w_1\cdots w_n}e_{n}$. The unconditional convergence of $\sum_n S_n x$ for any $x\in\mathcal D$ follows from the convergence of  $\sum_n (w_1\dots w_n)^{-p}$ whereas $B_{\mathbf w}^n x=0$ for any $x\in\mathcal D$ and any $n\geq 1$! Moreover, since $(\alpha_n)$ is dense
in $\mathbb K$, $\textrm{span}(S_k x_n;\ k\geq 0,n\geq 1)$ is dense in $X$. As observed above, $\sum_{k\geq 0} B_{\mathbf w}^k x_{n_k}=\alpha_{n_0}e_0$. 
Regarding $\sum_{k<0}S_{-k}x_{n_k}$, we just write
\begin{eqnarray*}
\left\|\sum_{k< 0}S_{-k} x_{n_k}\right\|&=&\left\|\sum_{k<0}\alpha_{n_k}e_{-k}\right\|\\
&=&\left(\sum_{k<0}\frac{|\alpha_{n_k}|^p}{(w_1\dots w_{-k})^p}\right)^{1/p}\\
&\leq&\left(\sum_{k<0}\frac{\omega_0(n_k)}{(w_1\dots w_{-k})^p}\right)^{1/p}.
\end{eqnarray*}
Since $\frac{\omega_0(n_{-1})}{w_1}\geq \frac1{w_1}$ and $\frac 1p\leq 1$, this in turn yields
$$\left\|\sum_{k< 0}S_{-k} x_{n_k}\right\|\leq C_{\mathbf w,p}\sum_{k<0}\frac{\omega_0(n_k)}{(w_1\dots w_{-k})^p}.$$

Thus, the assumptions of Theorem \ref{THMMAINTCL} are satisfied. Moreover, when $w_1\cdots w_n\geq Cn^{\alpha}$,
it is easy to check that, for any $x\in\mathcal D$, $\|T^n x\|=O(n^{-\alpha})$ and $\|S_n x\|=O(n^{-\alpha})$, allowing
to apply the results regarding the sequence of covariances and the central limit theorem.
\end{proof}
\begin{example}\label{EXBS}Let us now apply Theorem \ref{THMTCLBS} to backward shifts on $\ell^2(\mathbb Z_+)$. Let $\alpha\in(0,1)$ and let $B_{\mathbf w}$  be the weighted backward shift
on $\ell^2(\mathbb Z_+)$ with weight sequence $w_n=\left(\frac n{n-1}\right)^{\alpha+1/2}$, $n\geq 2$, $w_1=1$.
An associated perfectly spanning
$\TT$-eigenvectorfield is 
$$E(\lambda)=\sum_{n\geq 0}\frac{\lambda^n}{w_1\dots w_n}e_n,$$
where $(e_n)$ is the standard basis of $\ell^2(\mathbb Z_+)$.
$E$ is $\alpha$-H\"older (see \cite{Dev13}). Thus we know that there exists an ergodic Gaussian measure $\mu$ on $\ell^2(\mathbb Z_+)$ such that, 
for any $f,g\in\mathcal P$, 
$$|\cov(f\circ B_{\mathbf w}^n,g)|\leq C_{f,g,\alpha}n^{-\alpha}.$$
On the other hand, we may also apply Theorem \ref{THMTCLBS} to get an ergodic measure $\mu$ such that, for any $f,g\in\mathcal P$, 
$$|\cov(f\circ B_{\mathbf w}^n,g)|\leq C_{f,g,\alpha}
\left\{
\begin{array}{ll}
n^{-2\alpha}&\textrm{ if }\alpha\in(0,1/2)\\
\displaystyle \frac{\log (n+1)}n&\textrm{ if }\alpha=1/2\\[2mm]
n^{-\alpha-1/2}&\textrm{ if }\alpha>1/2.
\end{array}\right.
$$
In all cases, the speed of mixing is better.
\end{example}

\subsection{Composition operators associated to parabolic automorphisms}
The unpleasant condition (iv) in Theorem \ref{THMMAINTCL} is useful to handle operators such that $\sum_n T^n x$ and $\sum_n S_n x$ converge unconditionally
whereas $\sum_n \|T^n x\|=+\infty$ or $\sum_n \|S_n x\|=+\infty$. This was the case for backward shift operators. Another example is given by composition operators
induced by a parabolic automorphism of the disk.

Let $\phi$ be a parabolic automorphism of the disk, namely $\phi$ is an automorphism with a unique boundary fixed point. The composition operator $C_\phi$ 
defined by $C_\phi(f)=f\circ \phi$ is bounded on the Hardy space $H^2(\DD)$. Since, for any countable set $D\subset\TT$, the linear span of
$\bigcup_{\lambda\in\TT\backslash D}\ker(C_\phi-\lambda)$ is dense in $H^2(\DD)$ (see \cite{BAYGRIPLMS}), there exists a $C_\phi$-invariant strongly mixing Gaussian measure $\mu$ on $H^2(\DD)$ with full support. Moreover, it is shown in \cite[Example 3.9]{BAYGRIPLMS} that $C_\phi$ admits a $\TT$-eigenvector field $E$ which is 
$\alpha$-H\"older, for any $\alpha<1/2$. Devinck's result ensures that, for any $f\in\mathcal P$, $\cov(f\circ T^n,f)=O(n^{-\alpha})$, $\alpha<1/2$.

As for backward shifts, we can go further.
\begin{example}\label{EXCO}
Let $\phi$ be a parabolic automorphism of the disk, let $\alpha<1$ and let $\omega:\mathbb N\to(1,+\infty)$ going to infinity. There exists a $C_\phi$-invariant
strongly mixing Borel probability measure $\mu$ on $H^2(\DD)$ with full support such that $E_\omega\subset L ^2 (X,\mathcal B,\mu)$ and, for any $f,g\in E_\omega$, 
$$\cov(f\circ T^n,g)=O(n^{1-2\alpha}).$$
\end{example}
\begin{proof}
Without loss of generality, we may assume $\phi(1)=1$. It is easier to work with the upper half-plane $\mathbb P_+=\{s\in\CC;\ \Im m(s)>0\}$
which is biholomorphic to $\DD$ via the Cayley map $\sigma(z)=i(1+z)/(1-z)$. We set $\mathcal H^2:=\{f\circ\sigma^{-1};\ f\in H^2(\DD)\}$. The norm
on $\mathcal H^2$ is given by 
$$\|F\|_2^2=\pi^{-1}\int_{\mathbb R}|F(t)|^2\frac{dt}{1+t^2}.$$
Moreover, $C_\phi$, acting on $H^2(\DD)$, is similar, via the Cayley map, to a translation operator $\tau_a(F)=F(\cdot+a)$ acting on $\mathcal H^2$, $a\in\RR^*$.
We shall assume $a=1$. 

Let $p\geq 1$ and let $\mathcal D$ be the set of all holomorphic polynomials satisfying $P(1)=P'(1)=\dots=P^{(2p)}(1)=0$. $\mathcal D$ is dense in $H^2(\DD)$
and any $P\in\mathcal D$ satisfies $|P(z)|\leq C_P |z-1|^{2p}$. Hence $Q=P\circ\sigma^{-1}\in\mathcal H^2$ satisfies $|Q(x)|\leq\frac{C_Q}{(1+x^2)^p}$ for any $x\in\RR$. 

Let now $\omega_0:\NN\to(1,+\infty)$ and let $\theta=\sqrt{\omega_0}$. Let finally $(Q_n)_{n\geq 1}$ be a dense sequence in $\mathcal H^2$ such that, for any $n\geq 1$ and any $x\in\RR$, 
$$|Q(x)|\leq \frac{\theta(n)}{(1+x^2)^p}.$$

We claim that the assumptions of Theorem \ref{THMMAINTCL} are satisfied with $T=\tau_1$, $S_n=\tau_{-n}$ and $x_n=Q_n$. The first three points
can be found e.g. in \cite[Theorem 6.14]{BM09}. Let us prove (iv). Let $(n_k)\subset\NN^{\ZZ}$ and let us majorize
\begin{eqnarray*}
\left\|\sum_{k\geq 0}\tau_k Q_{n_k}\right\|_2^2&\leq&\int_{\mathbb R}\left(\sum_{k\geq 0}\frac{\theta(n_k)}{(1+(t+k)^2)^p}\right)^2\frac{dt}{1+t^2}\\
&\leq&2\sum_{k\geq j\geq 1}\int_{\mathbb R}\frac{\theta(n_k)\theta(n_j)}{(1+(t+k)^2)^p(1+(t+j)^2)^p(1+t^2)}dt.
\end{eqnarray*}
For $k,j\geq 1$, we say that $k\sim j$ provided $[-k-k^{1/4},-k+k^{1/4}]\cap [-j-j^{1/4},-j+j^{1/4}]\neq\varnothing$. It is important to notice that
\begin{eqnarray}\label{EQJKCOMPARABLE}
k\sim j\implies C^{-1}k\leq j\leq Ck.
\end{eqnarray}
Let us consider $k\geq j\geq 1$ and let us first assume that $k\nsim j$. We split the integral over $\mathbb R$ into three integrals: over $[-k-k^{1/4},-k+k^{1/4}]$, over
$[-j-j^{1/4},-j+j^{1/4}]$ and outside the two previous intervals. On $[-k-k^{1/4},-k+k^{1/4}]$, we know that $(t+j)^2\geq j^{1/2}$ (since $j\nsim k$), so that
$$\int_{-k-k^{1/4}}^{-k+k^{1/4}}\frac{dt}{(1+(t+k)^2)^p(1+(t+j)^2)^p(1+t^2)}\leq\frac{Ck^{1/4}}{1\times j^{p/2}\times k^2}\leq \frac{C}{k^{7/4}j^{7/4}}$$
provided $p\geq 7/2$. A similar estimation holds true on $[-j-j^{1/4},-j+j^{1/4}]$. On the remaining part of $\mathbb R$, we have both $(t+j)^2\geq j^{1/2}$ and
$(t+k)^2\geq k^{1/2}$. Since $\int_{\RR}dt/(1+t^2)<+\infty$, we finally get, provided $k\nsim j$, 
$$\int_{\mathbb R}\frac{dt}{(1+(t+k)^2)^p(1+(t+j)^2)^p(1+t^2)}\leq\frac C{k^{7/4}j^{7/4}}.$$
Suppose now $k\sim j$. We split the integral over $\mathbb R$ into two integrals: over $[-k-k^{1/4},-j+k^{1/4}]$ (recall that $k\geq j$) and outside this interval. 
Outside the interval, the estimation of the integral is very similar. Inside the interval, we cannot control both $(t+k)^2$ and $(t+j)^2$, whereas $(1+t^2)\geq Ck^2$. Observe also that the length of the interval
is controlled by $2k^{1/4}$. This leads to
$$\int_{\mathbb R}\frac{dt}{(1+(t+k)^2)^p(1+(t+j)^2)^p(1+t^2)}\leq\frac{C}{k^{7/4}}.$$
Summing over all possible values of $k$ and $j$ and using the previous estimates, we find 
\begin{eqnarray*}
\left\|\sum_{k\geq 1}\tau_k Q_{n_k}\right\|_2^2&\leq&C\sum_{k,j\geq 1}\frac{\theta(n_k)\theta(n_j)}{k^{7/4}j^{7/4}}+C\sum_{\substack{k\geq j\geq 1\\k\sim j}}\frac{\theta(n_k)\theta(n_j)}{k^{7/4}}\\
&\leq&C\left(\sum_{k\geq 1}\frac{\theta(n_k)}{k^{7/4}}\right)^2+C\sum_{\substack{k, j\geq 1\\k\sim j}}\frac{\theta(n_k)^2}{k^{7/4}}+C\sum_{\substack{k, j\geq 1\\k\sim j}}\frac{\theta(n_j)^2}{j^{7/4}}
\end{eqnarray*}
where the last inequality follows from (\ref{EQJKCOMPARABLE}). We then observe that, for a fixed $k\geq 1$, there is at most $Ck^{1/4}$
integers $j$ with $j\sim k$. Hence, 
$$\left\|\sum_{k\geq 1}\tau_k Q_{n_k}\right\|_2^2\leq C\left(\sum_{k\geq 1}\frac{\theta(n_k)}{k^{7/4}}\right)^2+C\sum_{k\geq 1}\frac{\theta(n_k)^2}{k^{3/2}}.$$
We take the square-root and observe that $\sqrt{a+b}\leq\sqrt a+\sqrt b$, $\sqrt c\leq c$ if $c\geq 1$. This gives
$$\left\|\sum_{k\geq 1}\tau_k Q_{n_k}\right\|_2\leq C\sum_{k\geq 1}\frac{\theta(n_k)^2}{k^{3/2}},$$
which proves (iv) with $\veps_k=k^{-3/2}$. We conclude the proof by showing that, for any $k,n\geq 1$, 
$$\|\tau_k Q_n\|_2\leq \frac{C\omega_0(n)}{k^\alpha}.$$
Indeed, 
$$\|\tau_k Q_n\|_2^2\leq \int_{\mathbb R}\frac{\theta(n)^2}{(1+(t+k)^2)^p}\frac{dt}{1+t^2}.$$
We argue as above. Let $\veps\in(0,1)$. We split the integral over $\mathbb R$ into the integral over $[-k-k^\veps,-k+k^\veps]$
and the integral over $\mathbb R\backslash [-k-k^\veps,-k+k^\veps]$. This yields
$$\|\tau_k Q_n\|_2^2\leq C\frac{\theta(n)^2}{k^{2-\veps}}+C\frac{\theta(n)^2}{k^{2\veps p}},$$
so that
$$\|\tau_k Q_n\|_2\leq C\omega_0(n)\left(\frac 1{k^{1-\veps/2}}+\frac 1{k^{\veps p}}\right)\leq \frac{C\omega_0(n)}{k^\alpha},$$
provided $1-\veps/2\geq \alpha$ and $\veps p\geq \alpha$.
\end{proof}

\begin{remark}
It is also possible to get a Central Limit Theorem for linear functionals in the context of parabolic composition operators.
See Section \ref{SECFURTHER}.
\end{remark}

\section{Central limit theorems - The proofs}\label{SECPROOF}

This section is devoted to the proof of Theorem \ref{THMMAINTCL}.

\subsection{How to prove a central limit theorem}

A central question in this paper is to find ways to prove central limit theorems for linear dynamical systems. This has been already studied in the general context of ergodic theory. Let $(\Omega,\mathcal A,\mu)$ be a probability space and let $T:\Omega\to\Omega$ be a bijective bimeasurable transformation preserving
the measure $\mu$. We also assume that $T$ is ergodic. Let $f\in L^2(\Omega)$, then $(f\circ T^i)_{i\in\mathbb Z}$ is a stationary process. 
We set $S_n(f)=\sum_{i=0}^{n-1}f\circ T^i$ and we say that $f$ satisfies the Central Limit Theorem (in short, CLT) if $\frac1{\sqrt n}S_n(f)$ converges in distribution
to a normal law.

To obtain sufficient conditions on a function $f$ so that the CLT holds, we shall use the \emph{martingale method} which was successfully used recently in various problems (see for instance \cite{LB99}, \cite{CLB05}, \cite{DiSi06}, \cite{Dup10}). This method goes back to Gordin in \cite{Gor69}. The basic idea is to try to approximate a given stationary sequence $f\circ T^n$
by a sequence which is a martingale difference sequence and to deduce the CLT for the given stationary sequence from the result for the martingale. 
An efficient sufficient condition was obtained by Maxwell and Woodroofe in \cite{MW00}. Let $(\mathcal F_i)_{i\in\mathbb Z}$ be a filtration with
$\mathcal F_i\subset T^{-1}\mathcal F_i=\mathcal F_{i+1}$. Let $\mathcal F_{\infty}$ be the smallest $\sigma$-algebra containing all the $\mathcal F_i$
and let $\mathcal F_{-\infty}=\bigcap_{i\in\mathbb Z}\mathcal F_i$. Maxwell and Woodroofe proved that, if $f\in L^2(\mathcal F_\infty)\ominus L^2(\mathcal F_{-\infty})$
is $\mathcal F_0$-measurable and 
$$\sum_{n=1}^{+\infty}\frac{\|\mathbb E(S_n(f)|\mathcal F_0)\|_2}{n^{3/2}}<+\infty,$$
then there exists a martingale difference sequence $(m\circ T^i)$ adapted to the filtration $(\mathcal F_i)$ such that $\|S_n(f-m\circ T^n)\|_2=o(\sqrt n)$. 
In particular, $f$ satisfies the CLT.

\smallskip

The monotone filtration $(\mathcal F_i)$ must be chosen in accordance with the transformation $T$; in a given concrete system the construction of this filtration
is difficult. In particular, the $\mathcal F_0$-measurability of $f$ (meaning that $(\mathcal F_i)$ is adapted to the sequence $(f\circ T^i)$) is a too restrictive
condition for the applications we have in mind. We will need a nonadapted version of the theorem of Maxwell and Woodroofe. This was done by Voln\'y in
\cite{Vol06}.
\begin{VOLNY}
Let $(\mathcal F_i)_{i\in\mathbb Z}$ be a filtration with $\mathcal F_i\subset T^{-1}(\mathcal F_i)=\mathcal F_{i+1}$. Let $f\in L^2(\mathcal F_\infty)\ominus L^2(\mathcal F_{-\infty})$ satisfying 
$$\sum_{n=1}^{+\infty}\frac{\|\mathbb E(S_n(f)|\mathcal F_0)\|_2}{n^{3/2}}<+\infty\textrm{ and }\sum_{n=1}^{+\infty}\frac{\|S_n(f)-\mathbb E(S_n(f)|\mathcal F_n)\|_2}{n^{3/2}}<+\infty.$$
Then $f$ satisfies the CLT.
\end{VOLNY}

%We will also need a variant of this result for noninvertible transformations. This variant will be suitable for backward shift operators. It can be found in \cite{TK05}.
%
%Let $T:(\Omega,\mathcal A,\mu)\to(\Omega,\mathcal A,\mu)$ be a noninvertible measure-preserving transformation which is ergodic. The Ruelle-Perron-Frobenius
%$P_T$ is defined by duality on $L^2$ by
%$$\int_X f\cdot g\circ Td\mu=\int_X P_T f\cdot gd\mu.$$
%This operator helps to quantify the defect of independance between $f$ and $f\circ T^n$. For instance, the correlations can be estimated by
%$$|\cov(f,f\circ T^n)|\leq \|P_T^n f\|_2\|f\|_2.$$
%Provided the partial sums $\sum_{k=0}^{n-1}P_T^kf$ do not grow too fast, $f$ satisfies the CLT.
%\begin{TYRAN}
%Let $f\in L^2_0(\Omega,\mathcal A,\mu)$ be such that 
%$$\sum_{n\geq 1}n^{-3/2}\left\|\sum_{k=0}^{n-1}P_T^k f\right\|_2<+\infty.$$
%Then $f$ satisfies the CLT. 
%\end{TYRAN}

It should be observed that Theorem A do not exclude that $S_n(f)/\sqrt n$ converges to a degenerate normal law, namely to a Dirac mass.
If we put stronger assumptions on $f$ (like the convergence of $\sum n |\cov(f,f\circ T^n)|$, see \cite{Liv96}), then this happens iff $f$ is a coboundary,
namely $f=g-g\circ T$ for some $g\in L^2$.

\subsection{The measure on $X$}\label{SECCONSTRUCTIONMEASURE}
From now on, we fix $\omega:[1,+\infty)\to\mathbb (1,+\infty)$ going to infinity. Without loss of generality, we may assume
that $\omega$ is nondecreasing. We consider $\omega_1:[1,+\infty)\to\mathbb (1,+\infty)$ tending to infinity and nondecreasing such that
 $$
 \left\{
 \begin{array}{l}
 \omega_1(k)^2=_{k\to+\infty}o(\omega(k))\\
 \omega_1(k+1)\leq 2\omega_1(k).
 \end{array}\right.$$
 
We then fix another function $\omega_0:[1,+\infty)\to\mathbb (1,+\infty)$
going to $+\infty$, nondecreasing, and such that
\begin{eqnarray}\label{EQOMEGA0}
\forall k,k'\geq 1,\ \omega_0(k+k')^{k+k'}\leq \omega_1(k)^k\omega_1(k')^{k'}.
\end{eqnarray}
For instance, we can set 
$$\omega_0(k)=\sqrt{\omega_1\left(\frac k2\right)}.$$
We claim that $\omega_0$ satisfies (\ref{EQOMEGA0}). Indeed, since $\omega_1$ is nondecreasing,
$$\forall k,k'\geq 1,\ \omega_1(k)^k\omega_1(k')^{k'}\geq \omega_1\left(\frac{k+k'}2\right)^{\frac{k+k'}2}.$$

\bigskip

The construction of the measure $\mu$ which appears in Theorem \ref{THMMAINTCL} follows \cite{MuPe13}. However, we will need to be  more careful during the construction because we want additional properties.
For convenience, throughout this section, for $k\in\mathbb N$, we shall denote by $T^{-k}x$ the vector $S_k x$, $x\in\mathcal D$.

The idea of Murillo-Arcila and Peris is to conjugate $T$ to a Bernoulli shift acting on $\NN^\ZZ$ and to transfer the ergodic properties
of this shift to $T$. We start from the sequence $(x_n)$ satisfying the assumptions of Theorem \ref{THMMAINTCL}. We may assume that $x_1=0$
and that $S_nx_1=0$ for all $n\geq 0$.
Let $(N_n)$ be an increasing sequence of positive integers with $N_{n+2}-N_{n+1}>N_{n+1}-N_n$
and satisfying, for any $n\in\NN$,
$$\left\|\sum_{k>N_n}T^k x_{m_k}\right\|+\left\|\sum_{k<-N_n}T^k x_{m_k}\right\|\leq \frac1{2^{n}}\textrm{ if }m_k\leq l,\ \textrm{ for }N_l<|k|\leq N_{l+1},\ l\geq n.$$

We define $K=\prod_{k\in\ZZ}F_k$ where
$$F_k=\{1,\dots,m\}\textrm{ if }N_m<|k|\leq N_{m+1}\textrm{ and }F_k=\{1\}\textrm{ if }|k|\leq N_1.$$
Let $K(s):=\sigma^{-s}(K)$, $s\in\ZZ$, where $\sigma:\NN^\ZZ\to\NN^\ZZ$ is the forward shift. The intertwining map $\phi$
is defined on $Z=\bigcup_{s\in\ZZ}K(s)$ by 
$$\phi\big((n_k)\big)=\sum_{k\in\ZZ}T^{k}x_{n_k}.$$
$\phi$ is well defined and continuous, and it satisfies on $Z$  the intertwining relation
\begin{eqnarray}
T\circ\phi=\phi\circ\sigma.\label{EQINTERTWINING}
\end{eqnarray}

Let us now construct on $\NN^\ZZ$ a measure $\bar\mu$ which is invariant for $\sigma$ and such that $\bar\mu(Z)=1$.
We fix a sequence $(p_l)$ of positive real numbers satisfying $\sum_l p_l=1$ and such that, setting
$$\beta_l=\left(\sum_{j=1}^{l} p_j\right)^{(N_{l+1}-N_{l})}>0,$$
then $\prod_{l\geq 1}\beta_l^2>0$.
This condition is satisfied provided $(p_l)$ converges sufficiently fast to zero. Then define $\bar\mu_k$ on $\NN$
by $\bar\mu_k(\{n\})=p_n$ and $\bar\mu$ on $\NN^\ZZ$ as $\bar\mu=\bigotimes_{k\in\ZZ}\bar\mu_k$.
It is shown in \cite{MuPe13} that $\bar\mu$ is a $\sigma$-invariant strongly mixing Borel probability measure on $\NN^\ZZ$ 
satisfying $\bar\mu(\ZZ)=1$. In particular, $\phi$ is defined almost everywhere on $\NN^\ZZ$ and (\ref{EQINTERTWINING})
is a.e. true. 

\smallskip

These properties can be transfered to $X$ by setting $\mu(A)=\bar\mu\big(\phi^{-1}(A)\big)$, $A\in\mathcal B(X)$. 
$\mu$ is a $T$-invariant strongly mixing Borel probability measure on $X$  (see  \cite{MuPe13}).
We have just to prove that $\mu$ has full support (in \cite{MuPe13}, this was done under the stronger assumption that $\mathcal D$ is dense).
Let $U$ be a nonempty open subset of $X$. Let $F$ be a finite subset of $\ZZ$ and let $(n_k)\subset \mathbb N^F$ be such that
$$y=\sum_{k\in F}T^k x_{n_k}\in U.$$
Let $n\in\mathbb N$ be such that $N_n>\max(\max F,-\min F)$ and $y+B(0,2^{-n})\subset U$. For $k\in[-n,n]\backslash F$, we set $n_k=1$. 
Then $U$ contains 
$$\left\{\sum_{k=-{N_n}}^{N_n}T^k x_{n_k}+\sum_{|k|>N_n}T^k x_{m_k};\ m_k\leq l\textrm{ for }N_l<|k|\leq N_{l+1},\ l\geq n\right\}.$$
Hence,
\begin{eqnarray*}
\mu(U)&\geq&\prod_{k=-N_n}^{N_n}\overline{\mu_k}(\{n_k\})\prod_{l=n}^{+\infty}\left(\prod_{N_l<|k|\leq N_{l+1}}\overline{\mu_k}(\{1,\dots,l\})\right)\\
&\geq&\prod_{k=-N_n}^{N_n}\overline{\mu_k}(\{n_k\})\prod_{l=n}^{+\infty}\beta_l^2>0.
\end{eqnarray*}

\smallskip

To ensure stronger properties than mixing, we will need additional assumptions on the sequence $(p_l)$.
We summarize these technical assumptions now, without further comments:
\begin{eqnarray}\label{EQPN2}
\sum_{m>l} p_m=o(p_l),
\end{eqnarray}
\begin{eqnarray}\label{EQPN3}
\forall l\geq 1,\ \forall k\geq 1,\ \sum_{m\geq l}\sqrt{p_m}\omega_0(m)^k\leq C\sqrt{p_l}\max\big(\omega_0(k)^k,\omega_0(l)^k\big),
\end{eqnarray}
\begin{eqnarray}\label{EQPN4}
\forall l\geq 1,\ \forall k\geq 1,\ \sum_{m\geq l}p_m\omega_0(m)^k\leq Cp_l\max\big(\omega_0(k)^k,\omega_0(l)^k\big)
\end{eqnarray}
where $C$ is some absolute constant. These conditions are satisfied if we require that the sequence $(p_l)$ decreases sufficiently fast to 0. 
This is clear for (\ref{EQPN2}) and also for (\ref{EQPN3}) if we restrict ourselves to $k\leq l$. If we now assume
$k>l$, then we can ensure (\ref{EQPN3}) by requiring that, for any $m\geq 1$,
$$\sqrt{p_{m+1}}\left(\omega_0(m+1)\right)^{m+1}\leq \frac 12\sqrt{p_m}\left(\omega_0(m)\right)^m.$$
Indeed, for any $l\geq 1$ and any $k>l$, one can decompose the sum into
$$\sum_{m\geq l}\sqrt{p_m}\omega_0(m)^k\leq \sum_{m=l}^{k}\sqrt{p_m}\omega_0(m)^{k}+\sum_{m=k+1}^{+\infty}\sqrt{p_m}\omega_0(m)^m.
$$
The first sum is bounded by $\left(\sum_{m\geq l}\sqrt{p_m}\right)\omega_0(k)^k$. The second sum is estimated as follows:
\begin{eqnarray*}
\sum_{m> k}\sqrt{p_m}\omega_0(m)^m&\leq &\sqrt{p_k}\omega_0(k)^k \left(\frac12+\frac14+\dots\right)\\
&\leq&\sqrt{p_l}\omega_0(k)^k.
\end{eqnarray*}
That we may also ensure (\ref{EQPN4}) follows along the same lines.

\bigskip

%Of course, the measure $\mu$ is more complicated than a Gaussian measure. However, with respect to this measure,
%the dynamics of $T$ is rather easy: $T$ behaves like a Bernoulli shift. This will allow us to compute
%the Ruelle-Perron-Frobenius operator and to apply Gordin's method. Beside the weighted shifts,
%it seems much more difficult to compute this operator when $T$ is considered as a self-map
%of $(X,\mathcal B,\nu)$, where $\nu$ is the Gaussian measure which appears in Theorem \ref{THMERGOBEST}. 
%This explains why the measure $\mu$ is probably more convenient to obtain a central limit
%theorem for orbits of operators.

\bigskip

We now show that our class of functions $E_\omega$ is contained in $L^2(X,\mathcal B,\mu)$.
\begin{lemma}\label{LEMXBMUISLARGE}
\begin{itemize}
\item[(a)] For any $d\geq 1$, $\|\cdot\|^d\in L^2(X,\mathcal B,\mu)$;
\item[(b)] $E_\omega\subset L^2(X,\mathcal B,\mu)$.
\end{itemize}
\end{lemma}
\begin{proof}
Let $d\geq 1$. The construction of the measure $\bar \mu$ ensures that, for almost every $(n_k)\subset\NN^{\ZZ}$, the series 
$\sum_{k\in\ZZ}T^k x_{n_k}$ is convergent. By Condition (iv), 
$$\|\phi\big((n_k)\big)\|\leq \sum_{k\in\ZZ}\veps_k\omega_0(n_k).$$
We expand the product and then use H\"older's inequality to get
\begin{eqnarray*}
\int_X \|x\|^{2d}d\mu(x)&\leq &\sum_{k_1,\dots,k_{2d}\in\mathbb Z}\int_{\mathbb N^{\mathbb Z}}\prod_{i=1}^{2d} \veps_{k_i}\omega_0(n_{k_i})d\bar \mu\big((n_k)\big)\\
&\leq&\sum_{k_1,\dots,k_{2d}\in\mathbb Z}\prod_{i=1}^{2d}\left(\int_{\mathbb N^{\mathbb Z}} \veps_{k_i}^{2d}\omega_0(n_{k_i})^{2d}d\bar \mu\big((n_k)\big)\right)^{1/2d}.
\end{eqnarray*}
Now,
$$\int_{\mathbb N^{\mathbb Z}} \veps_{k_i}^{2d}\omega_0(n_{k_i})^{2d}d\bar \mu\big((n_k)\big)\leq \veps_{k_i}^{2d}\sum_{l\geq 1} p_l\omega_0(l)^{2d}\leq C\veps_{k_i}^{2d}\omega_0(2d)^{2d},$$
where we have used (\ref{EQPN4}). Coming back to the $L^2$-norm
of $\|\cdot\|^d$, we finally obtain
\begin{eqnarray*}
\int_X \|x\|^{2d}d\mu(x)&\leq&C\omega_0(2d)^{2d}\sum_{k_1,\dots,k_{2d}\in\mathbb Z}\prod_{i=1}^{2d} \veps_{k_i}\\
&\leq&C\omega_0(2d)^{2d}\left(\sum_{k\in\ZZ}\veps_k\right)^{2d}\\
&\leq&C^{2d}\omega_0(2d)^{2d}.
\end{eqnarray*}
To prove (b), we start from a function $f\in E_\omega$ and we observe that
$$\|D^{\kappa}f(0)(x,\dots,x)\|\leq \|D^\kappa f(0)\|\times\|x\|^{\kappa}$$
so that, from the proof of the first point, we deduce
$$\|D^\kappa f(0)\|_{L^2}\leq C^\kappa\omega_0(\kappa)^\kappa  \|D^\kappa f(0)\|.$$
Since $\sup_{\kappa}\|D^\kappa f(0)\|\omega(\kappa)^\kappa<+\infty$ and $\omega_0(\kappa)=_{+\infty}o\big(\omega(\kappa)\big)$,  the series $\sum_\kappa \frac{D^\kappa f(0)(x,\dots,x)}{\kappa !}$ is convergent
in $L^2(X,\mathcal B,\mu)$, showing that $E_\omega\subset L^2(X,\mathcal B,\mu)$.
\end{proof}

\begin{remark}
If $f:X\to\mathbb R$ is infinitely differentiable at $0$, satisfies for any $x\in X$ 
$$f(x)=\sum_{\kappa=0}^{+\infty}\frac{D^\kappa f(0)}{\kappa !}(x,\dots,x),$$
and verifies moreover that, for any $R>0$, 
$$\sup_{\kappa>0}\|D^\kappa f(0)\|R^\kappa<+\infty,$$
then it is not hard to show that one can construct a function $\omega:\mathbb N\to(1,+\infty)$ going to $+\infty$
and such that $f\in E_\omega$. Under the assumptions of Theorem \ref{THMMAINTCL}, this means that one can define
a measure $\mu_f$ on $X$ such that the sequence $(f\circ T^n)$ satisfies the central limit theorem in 
$L^2(X,\mathcal B,\mu_f)$. However, to get a single measure $\mu$ which works for a large class of functions, 
we have to fix $\omega$ before.
\end{remark}

\subsection{An orthonormal basis of $L^2(X,\mathcal B,\mu)$}
In this section, we describe an orthonormal basis of $L^2(X,\mathcal B,\mu)$ or, equivalently,
an orthonormal basis of $L^2(\NN^\ZZ,\bar \mu)$. There is a usual way to do so. 
Suppose that $(e_l)_{l\geq 0}$ is an orthonormal basis of $\ell^2\big(\NN,(p_n)\big)$ with $e_0=1$.
Let $\NNinf=\bigcup_{k\geq 1}\NN^k$, $\ZZinf=\bigcup_{k\geq 1}\ZZ^k$ et
$\ZZinfinf=\{\mathbf j=(j_1,\dots,j_r)\in\ZZinf;\ j_1<j_2<\dots<j_r\}$.
For $\bf j\in\ZZinf$, denote by $|\bf j|$ the unique positive integer $r$ such that
$\mathbf j\in\ZZ^r$. 
Then, for $\mathbf l=(l_1,\dots,l_r)\in\NN^r$ and $\mathbf j=(j_1,\dots,j_r)\in\ZZ^r$
with $j_1<\dots<j_r$, define
$$e_{\mathbf l,\mathbf j}=e_{l_1,j_1}\times\dots\times e_{l_r,j_r},$$
where $e_{l,j}$ is a copy of $e_l$ on the $j$-th coordinate, namely 
$e_{l,j}\big((n_k)\big)=e_l(n_j)$. It is well known that 
$\{e_{\mathbf l,\mathbf j};\mathbf l\in\NNinf,\mathbf j\in\ZZinfinf,|\mathbf l|=|\mathbf j|\}\cup\{1\}$ is an orthonormal basis of 
$L^2(\NN^\ZZ,\bar \mu)$. Thus we just need to concentrate on the choice of $(e_l)$. 
We construct a triangular orthonormal basis.
\begin{lemma}\label{LEMORTHONORMALBASIS}
There exists an orthonormal basis $(e_l)$ of $\ell^2\big(\NN,(p_n)\big)$ with $e_0=1$ and, for 
$l\geq 1$,
$$e_l(u)=\left\{
\begin{array}{ll}
0&\textrm{ if }u<l\\
\dis \frac{1}{\sqrt{p_l}}\times\frac{\sqrt{\sum_{m>l}p_m}}{\sqrt{p_l+\sum_{m>l} p_l}}&\textrm{ if }u=l\\
\dis -\sqrt{p_l}\times \frac 1{\sqrt{p_l\left(\sum_{m>l} p_m\right)+\left(\sum_{m>l} p_m\right)^2}}&\textrm{ if }u>l.
\end{array}\right.$$
\end{lemma}
\begin{proof}
Given these formula, we just need to verify that this is an orthonormal basis. 
We observe that $(e_l)_{l\geq 0}$ is obtained by orthonormalization of the basis
$(1,1,\dots)$, $(0,1,1,\dots)$, $( 0,0,1,1,\dots),\dots$.
\end{proof}
Of course, the exact values of $e_l(u)$ are not very appealing. In the sequel, we will just need the following
estimations, which are satisfied thanks to (\ref{EQPN2}):
$$\left\{
\begin{array}{rcl}
| e_l(l)|&\leq&\displaystyle \frac C{\sqrt{p_l}}\\[0.4cm]
| e_l(u)|&\leq&\displaystyle \frac C{\sqrt{p_{l+1}}}\quad\textrm{ if }u>l.
\end{array}\right.
$$
Let us point out that this orthonormal basis behaves very well with respect to $\sigma$. Indeed, it is easy to check that 
\begin{eqnarray}
e_{\mathbf l,\mathbf j}\circ\sigma=e_{\mathbf l,\mathbf j-1}.\label{EQBASISSIGMA}
\end{eqnarray}
Let us also mention the following property which is the key for our forthcoming estimations. The $\ell^2\big(\mathbb N,(p_n)\big)$-norm
of each $e_l$ is equal to $1$. However, the $\ell^1\big(\mathbb N,(p_n)\big)$-norm of $e_l$ goes very quickly to zero: it behaves 
like $\sqrt{p_l}$.

\subsection{The Fourier coefficients}
In this subsection, we control the Fourier coefficients of a homogeneous polynomial. This will be the key point to
control later the behaviour of the sequence of the covariances $\cov(f\circ T^n,g)$ for $f,g\in E_\omega$.
We first observe that many Fourier coefficients of a homogeneous polynomial are equal to zero. 
\begin{lemma}\label{LEMCOVARIANCE2}
Let $P(x)=Q(x,\dots,x)$ be a homogeneous polynomial of degree $d$. 
Let $\mathbf j\in\ZZinfinf$, $\mathbf l\in\NNinf$ with $|\mathbf j|=|\mathbf l|>d$. Then
$$\langle e_{\mathbf l,\mathbf j},P\circ\phi \rangle_{L^2(\NN^\ZZ,\bar \mu)}=0.$$
\end{lemma}
\begin{proof}
Let us introduce some notations which will also be useful for the next lemma. 
We write $\mathbf j=\{j_1,\dots,j_r\}$, $\mathbf l=\{l_1,\dots,l_r\}$ and we set
\begin{eqnarray*}
\phi_{\mathbf j}\big((n_k)\big)&=&\phi\big((n_k)\big)-T^{j_1}(x_{n_{j_1}})-\dots-T^{j_r}(x_{n_{j_r}})\\
\widehat{\mathbb N_{\mathbf j}}&=&\mathbb N\backslash\{j_1,\dots,j_r\}\\
\bar{\nu}_{\mathbf j}&\textrm{is}&\textrm{the projection of $\bar \mu$ onto $\widehat{\mathbb N_{\mathbf j}}$}.
\end{eqnarray*}
Moreover, for $u_0,\dots,u_r$ nonnegative integers with $u_0+\dots+u_r=d$, we set
\begin{eqnarray*}
Q_{u_0,\dots,u_r}\big((n_k)\big)&=&Q\big(\underbrace{\phi_{\mathbf j}\big((n_k)\big),\dots,\phi_{\mathbf j}\big((n_k)\big)}_{u_0\textrm{ times}},\underbrace{T^{j_1}x_{n_{j_1}},\dots,T^{j_1}x_{n_{j_1}}}_{u_1\textrm{ times}},\dots,\\
&&\quad\quad\quad\underbrace{T^{j_r}(x_{n_{j_r}}),\dots,T^{j_r}(x_{n_{j_r}})}_{u_r\textrm{ times}}\big).
\end{eqnarray*}
The $d-$linearity of $Q$ yields
\begin{eqnarray*}\displaystyle\langle e_{\mathbf l,\mathbf j},P\circ\phi\rangle=\sum_{\substack{u_0+\dots+u_r=d\\u_i\geq 0}}
\binom{d}{u_0,\dots,u_r}\int_{\NN^\ZZ}e_{\mathbf l,\mathbf j}\big((n_k)\big)Q_{u_0,\dots,u_r}\big((n_k)\big)\dmunk.
\end{eqnarray*}
Since $r>d$, in each term of the sum, one of the $u_1,\dots,u_r$, say $u_k$, is equal to zero. We use Fubini's theorem and we integrate first with respect to the $j_k$-th coordinate. We get zero since
$$\int_{\NN}e_{l,j}(n)d\bar\mu_j(n)=0,\textrm{ for any }l\in\NN\textrm{ and any }j\in\ZZ.$$
\end{proof}
The nonzero coefficients will be estimated thanks to the following lemma.
\begin{lemma}\label{LEMCOVARIANCE3}
Let $P(x)=Q(x,\dots,x)$ be a homogeneous polynomial of degree $d$.
Let $\mathbf j\in\ZZinfinf$, $\mathbf l\in\NNinf$ with $|\mathbf j|=|\mathbf l|=r\leq d$. Then
$$\left|\langle e_{\mathbf l,\mathbf j},P\circ\phi\rangle_{L^2(\NN^\ZZ,\bar \mu)}\right|\leq 
\frac{C^d r^d\|Q\|\omega_0(d)^{d}\sup_i \omega_0(l_i)^d}{(1+|j_1|)^{\alpha}\dots (1+|j_r|)^\alpha}\sqrt{p_{l_1}\dots p_{l_r}}.$$
\end{lemma}
\begin{proof}
We keep the same notations so that 
$$\begin{array}{rl}
\displaystyle \langle e_{\mathbf l,\mathbf j},P\circ\phi\rangle=\sum_{\substack{u_0+\dots+u_r=d\\u_i\geq 0}}\binom{d}{u_0,\dots,u_r}&\displaystyle\int_{\NN^\ZZ}e_{\mathbf l,\mathbf j}\big((n_k)\big)Q_{u_0,\dots,u_r}\big((n_k)\big)\\
& \displaystyle\quad\quad d\bar\nu_{\mathbf j}\big((n_k)\big)d\bar\mu_{j_1}(n_{j_1})\dots d\bar\mu_{j_r}(n_{j_r}).
\end{array}$$
As in the previous lemma we have just to consider the terms in the sum such that $u_1\geq 1,\dots,u_r\geq 1$. Moreover,
$$\big|Q_{u_0,\dots,u_r}\big((n_k)\big)\big|\leq \|Q\|\times\|\phi_{\mathbf j}\big((n_k)\big)\|^{u_0}\prod_{i=1}^r \|T^{j_i}(x_{n_{j_i}})\|^{u_i}.$$
We get
\begin{eqnarray*}
\left|\langle e_{\mathbf l,\mathbf j},P\circ\phi\rangle\right|&\leq&\|Q\|\sum_{\substack{u_0+\dots+u_r=d\\u_i\geq 1}}\binom{d}{u_0,\dots,u_r}
\int_{\widehat{\mathbb N_{\mathbf j}}}\|\phi_{\mathbf j}\big((n_k)\big)\|^{u_0}d\bar{\nu_{\mathbf j}}\big((n_k)\big)\times\\
&&\quad\quad \prod_{i=1}^r
\int_{\mathbb N}|e_l(n)|\|T^{j_i}(x_n)\|^{u_i}d\bar{\mu_0}(n).
\end{eqnarray*}
The first integral may be  handled exactly like in Lemma \ref{LEMXBMUISLARGE} and we get
$$\int_{\widehat{\mathbb N_{\mathbf j}}}\|\phi_{\mathbf j}\big((n_k)\big)\|^{u_0}d\bar\nu_{\mathbf j}\big((n_k)\big)\leq C^{u_0}\omega_0(u_0)^{u_0}.$$
The estimation of the other integrals needs the properties of the sequences $(p_l)$ and $(e_l)$. Indeed, for $u\geq 1$, $j\in\mathbb Z$ 
and $l\in\mathbb N$, one can write 
\begin{eqnarray*}
\int_{\NN} |e_l(n)|\|T^{j}(x_n)\|^u d\bar\mu_0(n)&=&\sum_{m\geq l}p_m\|T^{j}(x_m)\|^u |e_l(m)|\\
&\leq&\frac{C\sqrt{p_l}}{(1+|j|)^{\alpha u}}\omega_0(l)^u+\frac C{\sqrt{p_{l+1}}(1+|j|)^{\alpha u}}\sum_{m\geq l+1}p_m\omega_0(m)^u.
\end{eqnarray*}
We then apply (\ref{EQPN4}) to get
\begin{eqnarray*}
\int_{\NN} |e_l(n)|\|T^{j}(x_n)\|^u d\bar\mu_0(n)
&\leq&\frac{C\sqrt{p_l}}{(1+|j|)^{\alpha u}}\omega_0(l)^u+\\
&&\quad\quad\frac C{\sqrt{p_{l+1}}(1+|j|)^{\alpha u}}p_{l+1}\max\left(\omega_0(l+1)^u,\omega_0(u)^u\right)\\
&\leq&\frac{C^u\sqrt{p_l}\max\big(\omega_0(l)^u,\omega_0(u)^u\big)}{(1+|j|)^\alpha}\\
&\leq&\frac{C^u\omega_0(u)^u\sqrt{p_l}\omega_0(l)^u}{(1+|j|)^\alpha}.
\end{eqnarray*}
Coming back to our original Fourier coefficient and using $u_0+\dots+u_r=d$, we get 
$$
\left|\langle e_{\mathbf l,\mathbf j},P\circ\phi\rangle\right|\leq C^d \|Q\|\omega_0(d)^d\sum_{\substack{u_0+\dots+u_r=d\\u_i\geq 1}}\binom{d}{u_0,\dots,u_r}
\frac{\omega_0(l_1)^{u_1}\dots\omega_0(l_r)^{u_r}}{(1+|j_1|)^\alpha\dots (1+|j_r|)^\alpha}\sqrt{p_{l_1}\dots p_{l_r}}.$$
We conclude by noting that the cardinal number of $\{(u_0,\dots,u_r)\in\mathbb N^{r+1};\ u_0+\dots+u_r=d,\ u_1,\dots,u_r\geq 1\}$ is less than or
equal to $2^d$ and that the multinomial coefficient is less than or equal to $(r+1)^d$.
\end{proof}
We give a first application of the previous lemmas. It deals with the sum of the Fourier coefficients of a given order for a function 
in $E_\omega$.
\begin{lemma}\label{LEMSUMFOURIER}
Let $f\in E_\omega$, $F=f\circ\phi=\sum_{\mathbf l,\mathbf j}a_{\mathbf l,\mathbf j}e_{\mathbf l,\mathbf j}$. Then, for any $M>1$,
there exists $C_{\omega_0,\omega,M}>0$ such that, for any
$r\geq 1$ and any $\mathbf j\in\mathbb \ZZinfinf$ with $|\mathbf j|=r$,
$$\sum_{|\mathbf l|=r}|a_{\mathbf l,\mathbf j}|\leq \frac{C_{\omega_0,\omega,M}\|f\|_\omega}{M^r (1+|j_1|)^\alpha\dots (1+|j_r|)^\alpha}.$$
\end{lemma}
\begin{proof}
By the above lemmas, 
\begin{eqnarray*}
\sum_{|\mathbf l|=r}|a_{\mathbf l,\mathbf j}|&\leq&\sum_{|\mathbf l|=r}\sum_{\kappa=0}^{+\infty}\frac{|\langle e_{\mathbf l,\mathbf j},D^\kappa f(0)\circ\phi\rangle|}{\kappa !}\\
&\leq&\sum_{\kappa=r}^{+\infty}C^{\kappa}r^\kappa \frac{\|D^\kappa f(0)\|\omega_0(\kappa)^\kappa}{\kappa !(1+|j_1|)^\alpha\dots (1+|j_r|)^\alpha}\sum_{|\mathbf l|=r}
{\sup_i \omega_0(l_i)^\kappa \sqrt{p_{l_1}\dots p_{l_r}}}.
\end{eqnarray*}
By Stirling's formula, $\frac{r^\kappa}{\kappa !}\leq C^\kappa$ for any $\kappa\geq r$, so that
\begin{eqnarray*}
\sum_{|\mathbf l|=r}|a_{\mathbf l,\mathbf j}|
&\leq&\sum_{\kappa=r}^{+\infty}\frac{C^{\kappa} \|D^\kappa f(0)\|\omega_0(\kappa)^\kappa}{(1+|j_1|)^\alpha\dots (1+|j_r|)^\alpha}\sum_{|\mathbf l|=r}
{\sup_i \omega_0(l_i)^\kappa \sqrt{p_{l_1}\dots p_{l_r}}}.
\end{eqnarray*}
Now, for $\kappa\geq r$, 
\begin{eqnarray*}
\sum_{|\mathbf l|=r}\sup_i \omega_0(l_i)^\kappa \sqrt{p_{l_1}\dots p_{l_r}}&\leq &\sum_{i=1}^r \sum_{|\mathbf l|=r}\omega_0(l_i)^\kappa 
\sqrt{p_{l_1}\dots p_{l_r}}\\
&\leq&\sum_{i=1}^r \sum_{l_i=1}^{+\infty}\omega_0(l_i)^\kappa \sqrt{p_{l_i}}\prod_{\substack{j=1\\j\neq i}}^r \sum_{l=1}^{+\infty}\sqrt{p_l}\\
&\leq&C^r \sum_{l\geq 1}\sqrt{p_l}\omega_0(l)^\kappa\\
&\leq&C^\kappa \omega_0(\kappa)^\kappa
\end{eqnarray*}
where we have used (\ref{EQPN3}).

To conclude, observe that, given any absolute constant $C$, there exists some constant $C_{\omega_0,\omega,M}$ such that, for  any $\kappa$, 
$C^\kappa \omega_0(\kappa)^{2\kappa}\leq C_{\omega_0,\omega,M}M^{-\kappa}\omega(\kappa)^\kappa.$
This yields
\begin{eqnarray*}
\sum_{|\mathbf l|=r}|a_{\mathbf l,\mathbf j}|&\leq&\frac{C_{\omega_0,\omega,M}}{(1+|j_1|)^\alpha\dots (1+|j_r|)^\alpha}
\sum_{\kappa=r}^{+\infty}\frac{\|D^\kappa f(0)\|\omega(\kappa)^\kappa }{M^\kappa }\\
&\leq&\frac{C_{\omega_0,\omega,M}\|f\|_\omega}{M^r(1+|j_1|)^\alpha\dots (1+|j_r|)^\alpha}.
\end{eqnarray*}
\end{proof}

\subsection{The sequence of covariances}
In this subsection, we prove the first part of Theorem \ref{THMMAINTCL}, that devoted to the behaviour
of the sequence of covariances. We begin with a first lemma which is an easy consequence
of (\ref{EQBASISSIGMA}).
\begin{lemma}\label{LEMCOVARIANCE1}
Let $F=\sum_{\mathbf j,\mathbf l}a_{\mathbf l,\mathbf j}e_{\mathbf l,\mathbf j},G=\sum_{\mathbf j,\mathbf l}b_{\mathbf l,\mathbf j}e_{\mathbf l,\mathbf j}\in L^2_0(\NN^\ZZ,\bar \mu)$. Then
$$\cov(F\circ \sigma ^p,G)=\sum_{\mathbf j,\mathbf l}a_{\mathbf l,\mathbf j}b_{\mathbf l,\mathbf j-p}.$$
\end{lemma}
\begin{proof}
By linearity and by (\ref{EQBASISSIGMA}), 
$$\cov(F\circ\sigma^p,G)=\sum_{\mathbf j,\mathbf k,\mathbf l,\mathbf m}a_{\mathbf l,\mathbf j}b_{\mathbf m,\mathbf k}\cov(e_{\mathbf l,\mathbf j-p},e_{\mathbf m,\mathbf k}).$$
Now, $e_{\mathbf l,\mathbf j-p}$ and $e_{\mathbf m,\mathbf k}$ are orthogonal, unless $\mathbf l=\mathbf m$ and $\mathbf k=\mathbf j-p$. This gives the lemma.
\end{proof}

Let us now start with $f,g\in E_\omega$. Without loss of generality, we may assume that they have zero mean. We set
$F=f\circ\phi$, $G=g\circ \phi$ so that $\cov(f\circ T^p,g)=\cov(F\circ\sigma^p,G)$. Let us write
$$F=\sum_{r=1}^{+\infty}\sum_{|\mathbf l|=|\mathbf j|=r}a_{\mathbf l,\mathbf j}e_{\mathbf l,\mathbf j}\textrm{ and }
G=\sum_{r=1}^{+\infty}\sum_{|\mathbf l|=|\mathbf j|=r}b_{\mathbf l,\mathbf j}e_{\mathbf l,\mathbf j},$$
and let us compute $\cov(F\circ\sigma^p,G)$ using Lemma \ref{LEMCOVARIANCE1}. We apply the estimations of the Fourier coefficients
given by Lemma \ref{LEMCOVARIANCE2} and Lemma \ref{LEMCOVARIANCE3}. Using Stirling's formula, we find,
for a given $\mathbf j\in\ZZinfinf$ with $|\mathbf j|=r$,
$$\sum_{|\mathbf l|=r}|a_{\mathbf l,\mathbf j}b_{\mathbf l,\mathbf j-p}|\leq\sum_{|\mathbf l|=r}\sum_{\kappa,\kappa'=r}^{+\infty}
\frac{\|D^\kappa f(0)\| \|D^{\kappa'}f(0)\|C^{\kappa+\kappa'}\omega_0(\kappa)^{\kappa}\omega_0(\kappa')^{\kappa'} \sup_i\omega_0(l_i)^{\kappa+\kappa '}p_{l_1}\dots p_{l_r}}
{(1+|j_1-p|)^\alpha\dots (1+|j_r-p|)^\alpha (1+|j_1|)^\alpha\dots (1+|j_r|)^\alpha}.$$
Now, for $\kappa,\kappa'\geq r$,
\begin{eqnarray*}
\sum_{|\mathbf l|=r}\sup_i\omega_0(l_i)^{\kappa+\kappa '}p_{l_1}\dots p_{l_r}&\leq&
\sum_{i=1}^r \sum_{|\mathbf l|=r}\omega_0(l_i)^{\kappa+\kappa'}p_{l_1}\dots p_{l_r}\\
&\leq&\sum_{i=1}^r \left(\sum_{l_i=1}^{+\infty}p_{l_i}\omega_0(l_i)^{\kappa+\kappa'}\right)\times\left(\sum_{l\geq 1}p_l\right)^{r-1}\\
&\leq&Cr\omega_0(\kappa+\kappa')^{\kappa+\kappa'}\\
&\leq&C^{\kappa+\kappa'}\omega_1(\kappa)^{\kappa}\omega_1(\kappa')
\end{eqnarray*}
(at this stage, we use the strange relation (\ref{EQOMEGA0}) satisfied by $\omega_0$ and also (\ref{EQPN4})). Hence,
\begin{eqnarray*}
\sum_{|\mathbf l|=r}|a_{\mathbf l,\mathbf j-p}b_{\mathbf l,\mathbf j}|&\leq&
 \left(\sum_{\kappa=r}^{+\infty}C^\kappa \|D^\kappa f(0)\|\omega_1(\kappa)^{2\kappa}\right)
\times
 \left(\sum_{\kappa=r}^{+\infty}{C^\kappa \|D^\kappa g(0)\|\omega_1(\kappa)^{2\kappa}}\right)\times\\
&&\quad\quad\quad\quad\frac1{(1+|j_1-p|)^\alpha\dots (1+|j_r-p|)^\alpha (1+|j_1|)^\alpha\dots (1+|j_r|)^\alpha} \\
&\leq&\frac{C\|f\|_\omega \|g\|_\omega}{(1+|j_1-p|)^\alpha\dots (1+|j_r-p|)^\alpha (1+|j_1|)^\alpha\dots (1+|j_r|)^\alpha} .
\end{eqnarray*}
We now sum over $\mathbf j$ with $|\mathbf j|=r$:
\begin{eqnarray*}
\sum_{|\mathbf j|=|\mathbf l|=r}|a_{\mathbf l,\mathbf j-p}b_{\mathbf l,\mathbf j}|&\!\leq\!&\sum_{j_1<j_2<\dots<j_r}
\frac{C\|f\|_\omega \|g\|_\omega} {(1+|j_1-p|)^\alpha\dots (1+|j_r-p|)^\alpha (1+|j_1|)^\alpha\dots (1+|j_r|)^\alpha}\\
&\!\leq\!&C\|f\|_\omega \|g\|_\omega\left(\sum_{j\in\mathbb Z}\frac1{(1+|j-p|)^\alpha (1+|j|)^\alpha}\right)^r.
\end{eqnarray*}

We split this last sum into three sums : $\sum_{j\leq -1}$, $\sum_{j=0}^p$ and $\sum_{j=p+1}^{+\infty}$ and we observe that the first sum and the last sum are equal. 
We first consider these sums. We get different estimates following the value of $\alpha$. When $\alpha>1$, it is easy to check that, for $j\geq p+1$,
$$\frac 1{(1+(j-p))^\alpha}-\frac 1{(1+j)^\alpha}=\frac{(1+j)^\alpha-(1+j-p)^\alpha}{(1+j)^\alpha (1+(j-p))^\alpha}\geq \frac{p^\alpha}{(1+j)^\alpha (1+(j-p))^\alpha}.$$
Thus,
\begin{eqnarray*}
\sum_{j\geq p+1}\frac 1{(1+j)^\alpha(1+(j-p))^\alpha}&\leq&p^{-\alpha}\sum_{j\geq p+1}\left(\frac 1{(1+(j-p))^\alpha}-\frac1{(1+j)^\alpha}\right)\\
&\leq&p^{-\alpha}\left(\frac 1{2^\alpha}+\dots+\frac{1}{(p+1)^\alpha}\right)\leq C_\alpha p^{-\alpha}.
\end{eqnarray*}
On the other hand, 
$$\sum_{j=0}^p\frac{1}{(1+j)^\alpha (1+(p-j))^\alpha}\leq C_\alpha p^{-\alpha}\sum_{j=0}^{[p/2]}\frac{1}{(1+j)^\alpha}\leq C_\alpha p^{-\alpha}.$$
We finally find
$$\sum_{|\mathbf j|=|\mathbf l|=r}|a_{\mathbf l,\mathbf j-p}b_{\mathbf l,\mathbf j}|\leq C_\alpha\|f\|_\omega \|g\|_\omega \left(\frac{C_\alpha}{p^{\alpha}}\right)^r.$$
Summing this over $r$, we get
$$|\cov(F\circ\sigma^p,G)|\leq \frac{C_\alpha}{p^\alpha}\|f\|_\omega\|g\|_\omega$$
provided $p^\alpha\geq 2C_\alpha$.

The case $\alpha=1$ is the easiest one. Indeed, in that case, for $j\geq p+1$ one can write
$$\frac1{1+(j-p)}-\frac1{1+j}=\frac p{(1+j)(1+j-p)}.$$
We then argue exactly as before. Suppose now that $\alpha\in(1/2,1)$. On the one hand
$$\sum_{j\geq p+1}\frac{1}{(1+j)^\alpha (1+(j-p))^\alpha}\leq \sum_{j=1}^p \frac 1{j^\alpha(j+p)^\alpha}+
\sum_{j=p+1}^{+\infty}\frac 1{j^\alpha(j+p)^\alpha}.$$
Now,
$$\sum_{j=1}^p \frac 1{j^\alpha(j+p)^\alpha}\leq\frac1{p^\alpha}\sum_{j=1}^p \frac1{j^\alpha}\leq C_\alpha p^{1-2\alpha}$$
whereas
$$\sum_{j=p+1}^{+\infty}\frac 1{j^\alpha(j+p)^\alpha}\leq \sum_{j=p+1}^{+\infty}\frac 1{j^{2\alpha}}\leq {C_\alpha}p^{1-2\alpha}.$$
On the other hand,
\begin{eqnarray*}
\sum_{j=0}^p \frac1{(1+j)^\alpha (1+(p-j))^\alpha}&\leq&\frac2{p^\alpha}+\sum_{j=1}^{p-1}\frac 1{j^\alpha (p-j)^\alpha}\\
&\leq&\frac2{p^\alpha}+p^{1-2\alpha}\times\frac1p\sum_{j=1}^{p-1}\frac 1{\left(\frac jp\right)^\alpha \left(1-\frac jp\right)^\alpha}
\end{eqnarray*}
We recognize a Riemann sum of the function $x\mapsto x^\alpha (1-x)^\alpha$, so that
$$\sum_{j=0}^p \frac1{(1+j)^\alpha (1+(p-j))^\alpha}\leq C_\alpha p^{1-2\alpha}.$$
We conclude exactly like for the other cases.

\subsection{Central limit theorem}

We now prove that the central limit theorem holds for $f\in E_\omega$ with zero mean. Throughout the proof,
we assume $\alpha>1$.
We set $F=f\circ \phi=\sum_{\mathbf l,\mathbf j}a_{\mathbf l,\mathbf j} e_{\mathbf l,\mathbf j}$. We will apply Theorem A with the filtration $(\mathcal F_i)_{i\in\mathbb Z}$ defined by $\mathcal F_i=\sigma^{-i}(\mathcal F_0)$ and 
$$\mathcal F_0=\dots\times\Omega\times\dots\times\Omega\times\mathcal P(\NN)\times\mathcal P(\NN)\times\dots,$$
where $\Omega=\{\varnothing,\NN\}$ and the first $\mathcal P(\NN)$ is at the $0$-th position. 
The filtration $(\mathcal F_i)_{i\in\mathbb Z}$ is increasing, with $\mathcal F_{\infty}=\mathcal P(\NN)^\ZZ$ and $\mathcal F_{-\infty}=\{\varnothing,\NN^\ZZ\}.$
Thus, $F\in L^2(\mathcal F_{\infty})\ominus L^2(\mathcal F_{-\infty})$. By (\ref{EQBASISSIGMA}), 
$$F\circ\sigma^p =\sum_{\mathbf l,\mathbf j}a_{\mathbf l,\mathbf j} e_{\mathbf l,\mathbf j-p}=\sum_{\mathbf l,\mathbf j}a_{\mathbf l,\mathbf j+p} e_{\mathbf l,\mathbf j}$$
so that
$$S_n(F)=\sum_{\mathbf l,\mathbf j}\sum_{p=0}^{n-1}a_{\mathbf l,\mathbf j+p} e_{\mathbf l,\mathbf j}.$$
If we take the conditional expectation, then we find
\begin{eqnarray*}
\mathbb E(S_n(F)|\mathcal F_0)&=&\sum_{\substack{\mathbf l,\mathbf j\\j_1\geq 0}}\sum_{p=0}^{n-1}a_{\mathbf l,\mathbf j+p} e_{\mathbf l,\mathbf j}\\
S_n(F)-\mathbb E(S_n(F)|\mathcal F_n)&=&\sum_{\substack{\mathbf l,\mathbf j\\j_1<-n}}\sum_{p=0}^{n-1}a_{\mathbf l,\mathbf j+p} e_{\mathbf l,\mathbf j}.
\end{eqnarray*}
During the proof, we will need the two following technical facts.
\begin{Fact1}
For any $n\geq 1$, 
$$\sum_{j\geq 0}\left(\sum_{p=0}^{n-1}\frac 1{(1+j+p)^\alpha}\right)^2\leq C_\alpha\max\left(n^{3-2\alpha},\log(n+1)\right).$$
\end{Fact1}

\begin{Fact2}
For any $n\geq 1$ and any $r\geq 1$,
\begin{eqnarray*}
\sum_{\substack{j_1<-n\\j_2,\dots,j_r\in\ZZ}}\left|\sum_{p=0}^{n-1}\frac 1{(1+|j_1+p|)^\alpha\dots (1+|j_r+p|)^\alpha}\right|^2&\leq& \sum_{j_1<-n}
\left(\sum_{p=0}^{n-1}\frac 1{(1+|j_1+p|)^\alpha}\right)^2\times\\
&&\quad\left(\sum_{j\in\ZZ}\frac 1{(1+|j|)^\alpha}\right)^{r-1}.
\end{eqnarray*}
\end{Fact2}

We postpone the proof of these two facts and we show that the conditions of Theorem A are satisfied. First,
\begin{eqnarray*}
\left\|\mathbb E(S_n(F)|\mathcal F_0)\right\|_2^2&=&\sum_{r\geq 1}\sum_{\substack{|\mathbf j|=r\\j_1\geq 0}}\sum_{|\mathbf l|=r}\left|\sum_{p=0}^{n-1}a_{\mathbf l,\mathbf j+p} \right|^2\\
&\leq&\sum_{r\geq 1}\sum_{\substack{|\mathbf j|=r\\j_1\geq 0}}\left(\sum_{p=0}^{n-1}\sum_{|\mathbf l|=r}|a_{\mathbf l,\mathbf j+p} |\right)^2.
\end{eqnarray*}
We apply Lemma \ref{LEMSUMFOURIER} with $M^2\geq 2\sum_{j\geq 1}\frac 1{(1+j)^{2\alpha}}.$ This yields
\begin{eqnarray*}
\left\|\mathbb E(S_n(F)|\mathcal F_0)\right\|_2^2&\leq&\sum_{r\geq 1}\frac{C_{\omega_0,\omega,\alpha}\|f\|_\omega^2}{M^{2r}}\sum_{0\leq j_1<\dots<j_r}
\left(\sum_{p=0}^{n-1}\frac 1{(1+j_1+p)^\alpha\dots (1+j_r+p)^\alpha}\right)^2\\
&\leq&\sum_{r\geq 1}\frac{C_{\omega_0,\omega,\alpha}\|f\|_\omega^2}{M^{2r}}\sum_{j_2,\dots,j_r\geq 1}\frac1{(1+j_2)^{2\alpha}\dots(1+j_r)^{2\alpha}}\times\\
&&\quad\quad
\sum_{j_1\geq 0}\left(\sum_{p=0}^{n-1}\frac 1{(1+j_1+p)^\alpha}\right)^2.
\end{eqnarray*}
We now apply Fact 1 to get 
\begin{eqnarray*}
\left\|\mathbb E(S_n(F)|\mathcal F_0)\right\|_2^2&\leq&\sum_{r\geq 1}\frac{C_{\omega_0,\omega,\alpha}\|f\|_\omega^2}{2^r}\times \max\big(n^{3-2\alpha},\log(n+1)\big)\\
&\leq&C_{\omega_0,\omega,\alpha}\|f\|_\omega^2 \max\big(n^{3-2\alpha},\log(n+1)\big).
\end{eqnarray*}
Since $(3-2\alpha)/2<1/2$, this yields the convergence of $\sum_{n\geq 1}\frac{\left\|\mathbb E(S_n(F)|\mathcal F_0)\right\|_2}{n^{3/2}}$. 

\smallskip

We now turn to the second sum. The beginning of the estimation is completely similar, except that we now apply Lemma \ref{LEMSUMFOURIER} with $M^2=2\sum_{j\in\mathbb Z}\frac 1{(1+| j|)^{\alpha}}.$ We thus obtain 
$$\left\|S_n(F)-\mathbb E(S_n(F)|\mathcal F_n)\right\|_2^2\leq \sum_{r\geq 1}\frac{C_{\omega_0,\omega,\alpha}\|f\|_\omega^2}{M^{2r}}\sum_{\substack{|\mathbf j|=r\\j_1<-n}}
\left(\sum_{p=0}^{n-1}\frac 1{(1+|j_1+p|)^\alpha\dots (1+|j_r+p|)^\alpha}\right)^2.$$
At this stage, we can no longer majorize $\frac1{(1+|j_k+p|)^\alpha}$ by $\frac1{(1+|j_k|)^\alpha}$  for $k\geq 2$ since
it is possible that $j_k\leq 0$. We use Fact 2 instead. It yields
\begin{eqnarray*}
\left\|S_n(F)-\mathbb E(S_n(F)|\mathcal F_n)\right\|_2^2&\leq&\sum_{r\geq 1}\frac{C_{\omega_0,\omega,\alpha}\|f\|_\omega^2}{M^{2r}}\sum_{j_1<-n}
\left|\sum_{p=0}^{n-1}\frac 1{(1+|j_1+p|)^\alpha}\right|^2\times\\
&&\quad\quad\left(\sum_{j\in\ZZ}\frac1{(1+|j|)^\alpha}\right)^{r-1}.
\end{eqnarray*}
The definition of $M$ and the changes of variables $j=-j_1-n$, $q=n-1-p$ imply
\begin{eqnarray*}
\left\|S_n(F)-\mathbb E(S_n(F)|\mathcal F_n)\right\|_2^2&\leq&\sum_{r\geq 1}\frac{C_{\omega_0,\omega,\alpha}\|f\|_\omega^2}{2^r}\sum_{j>0}
\left|\sum_{q=0}^{n-1}\frac 1{(1+j+q)^\alpha}\right|^2\\
&\leq&C_{\omega_0,\omega,\alpha}\|f\|_\omega^2\max\big(n^{3-2\alpha},\log(n+1)\big)
\end{eqnarray*}
where the last inequality follows from Fact 1. As before,
$$\sum_{n\geq 1}\frac{\left\|S_n(F)-\mathbb E(S_n(F)|\mathcal F_n)\right\|_2}{n^{3/2}}<+\infty.$$
Hence, $F$ satisfies the CLT.

\begin{proof}[Proof of Fact 1]
We first observe that there exists $C_\alpha>0$ such that, for any $j\geq 0$, 
$$\sum_{p=0}^{n-1}\frac1{(1+j+p)^\alpha}\leq C_\alpha\frac1{(1+j)^{\alpha-1}}.$$
Hence, 
\begin{eqnarray*}
\sum_{j=0}^n \left(\sum_{p=0}^{n-1}\frac1{(1+j+p)^\alpha}\right)^2&\leq&C_\alpha\sum_{j=0}^n \frac1{(1+j)^{2\alpha-2}}\\
&\leq&C_\alpha \begin{cases}
1&\textrm{ provided } \alpha>3/2\\
\log(n+1)&\textrm{ provided } \alpha=3/2\\
n^{3-2\alpha}&\textrm{ provided } \alpha\in(1,3/2).
\end{cases}
\end{eqnarray*}
For the remaining part of the sum, we just write
$$\sum_{p=0}^{n-1}\frac 1{(1+j+p)^\alpha}\leq\frac n{j^\alpha}$$
so that
$$\sum_{j>n}\left(\sum_{p=0}^{n-1}\frac 1{(1+j+p)^\alpha}\right)^2\leq n^2\sum_{j>n}\frac1{j^{2\alpha}}\leq C_\alpha n^{3-2\alpha}.$$
\end{proof}
\begin{proof}[Proof of Fact 2]
Let 
$$S=\sum_{\substack{j_1<-n\\j_2,\dots,j_r\in\mathbb Z}}\left(\sum_{p=0}^{n-1}\frac 1{(1+|j_1+p|)^\alpha\dots (1+|j_r+p|)^\alpha}\right)^2.$$
We expand the square to get
\begin{eqnarray*}
S&=&\sum_{\substack{j_1<-n\\j_2,\dots,j_r\in\mathbb Z}}\left(\sum_{p=0}^{n-1}\frac1{(1+|j_1+p|)^{2\alpha}\dots (1+|j_r+p|)^{2\alpha}}+\right.\\
&&\quad\quad\quad\quad\quad\quad \left.2\sum_{0\leq p<q\leq n-1}
\frac1{(1+|j_1+p|)^\alpha\dots (1+|j_r+q|)^\alpha}\right).
\end{eqnarray*}
We put the sum over $j_2,\dots,j_r$ inside and we observe that, for a fixed $p\in\mathbb Z$, 
\begin{eqnarray*}
\sum_{j_2,\dots,j_r\in\mathbb Z}\frac1{(1+|j_2+p|)^{2\alpha}\dots (1+|j_r+p|)^{2\alpha}}&=&\sum_{j_2,\dots,j_r\in\mathbb Z}\frac 1{
(1+|j_2|)^{2\alpha}\dots (1+|j_r|)^{2\alpha}}\\
&=&\left(\sum_{j\in\mathbb Z}\frac 1{(1+|j|)^{2\alpha}}\right)^{r-1}.
\end{eqnarray*}
Similarly, for a fixed $(p,q)\in\mathbb Z^2$, since $(1+|j_k+q|)\geq 1$, we get
\begin{eqnarray*}
\sum_{j_2,\dots,j_r\in\mathbb Z}\frac1{(1+|j_2+p|)^{\alpha}\dots (1+|j_r+q|)^{\alpha}}&\leq&\sum_{j_2,\dots,j_r\in\mathbb Z}\frac 1{
(1+|j_2+p|)^{\alpha}\dots (1+|j_r+p|)^{\alpha}}\\
&\leq&\sum_{j_2,\dots,j_r\in\mathbb Z}\frac 1{
(1+|j_2|)^{\alpha}\dots (1+|j_r|)^{\alpha}}\\
&=&\left(\sum_{j\in\mathbb Z}\frac 1{(1+|j|)^\alpha}\right)^{r-1}.
\end{eqnarray*}
Hence,
\begin{eqnarray*}
S&\leq&\sum_{j_1<-n}\left(\sum_{p=0}^{n-1}\frac1{(1+|j_1+p|)^{2\alpha}}+2\sum_{0\leq p<q\leq n-1}
\frac1{(1+|j_1+p|)^\alpha (1+|j_1+q|)^\alpha}\right)\times\\
&&\quad\quad\left(\sum_{j\in\mathbb Z}\frac 1{(1+|j|)^\alpha}\right)^{r-1}\\
&\leq&\sum_{j_1<-n}\left(\sum_{p=0}^{n-1}\frac 1{(1+|j_1+p|)^\alpha}\right)^2\times\quad\left(\sum_{j\in\ZZ}\frac 1{(1+|j|)^\alpha}\right)^{r-1}  
\end{eqnarray*}

\end{proof}

\section{Further remarks}\label{SECFURTHER}

\subsection{Unconditional convergence}
In the statement of Theorem \ref{THMMAINTCL}, Condition (iv) is not very pleasant. During the proof it is used at two places: in Lemma \ref{LEMXBMUISLARGE} and in Lemma \ref{LEMCOVARIANCE3}. We can delete this assumption if we accept to work only with polynomials instead of functions in $E_\omega$.
\begin{theorem}\label{THMMAINTCLPOLY}
Let $T\in\LX$. Suppose that there exist a dense set $\mathcal D\subset X$ and a  sequence of maps $S_n:\mathcal D\to X$, $n\geq 0$, such that,
for any $x\in X$,
\begin{enumerate}[(i)]
\item $\sum_{n\geq 0} T^nx$ converges unconditionally;
\item $\sum_{n\geq 0} S_n x$ converges unconditionally;
\item $T^n S_n x= x$ and $T^mS_n x=S_{n-m}x$ for any $n>m$;
\end{enumerate}
Then there exists a $T$-invariant strongly mixing Borel probability measure $\mu$ on $X$
with full support such that $\mathcal P \subset L^2(X,\mathcal B,\mu)$.\\
Suppose moreover that there exists $\alpha>1/2$ such that, for any $x\in\mathcal D$, $\|T^n x\|=O(n^{-\alpha})$ 
and $\|S_n x\|=O(n^{-\alpha})$. 
Then, for any $f,g\in \mathcal P$,
$$|\cov(f\circ T^n,g)|\leq C_{f,g,\alpha}
\left\{
\begin{array}{ll}
n^{1-2\alpha}&\textrm{ if }\alpha\in(1/2,1)\\
\displaystyle \frac{\log (n+1)}n&\textrm{ if }\alpha=1\\
n^{-\alpha}&\textrm{ if }\alpha>1.
\end{array}\right.
$$
If $\alpha>1$, then for any $f\in \mathcal P$ with zero mean, the sequence $\dis\frac{1}{\sqrt n}(f+\dots+f\circ T^{n-1})$
converges in distribution to a Gaussian random variable of zero mean and finite variance.
\end{theorem}
\begin{proof}
We just point out the main differences with the proof of Theorem \ref{THMMAINTCL}. Let $\omega_0:\mathbb N\to (1,+\infty)$ be any nondecreasing function
going to $+\infty$. Let $(x_n)_{n\in\mathbb N}$ be a dense sequence in $\mathcal D$ with $x_1=0$ and $\|T^k x_n\|\leq \omega_0(n)$ for any $n\geq 1$ and any $k\in\mathbb Z$. We may also ask, if we want to prove the second part of Theorem \ref{THMMAINTCL2}, that for any $n\geq 1$ and any $k\in\mathbb Z$, 
$$\|T^k x_n\|\leq\frac{\omega_0(n)}{(1+|k|)^\alpha}.$$
We construct the measure exactly like in Section \ref{SECCONSTRUCTIONMEASURE}, except that we require that the sequence $(p_l)$ also satisfies
$$\forall d\geq 1,\ \forall l\geq 2d,\ \omega_0(l)^{2d}\leq p_l^{-1/2},$$
$$\forall d\geq 1,\ \sum_{l\geq 1}(N_{l+1}-N_l)p_l^{1/4d}<+\infty.$$
To prove that $\mathcal P\subset L^2(X,\mathcal B,\mu)$, it suffices to show that, for any $d\geq 1$, $\|\cdot\|^d\in L^2(X,\mathcal B,\mu)$. Now, 
by the triangle inequality,
\begin{eqnarray*}
\int_X \|x\|^{2d}&=&\int_{\NN^{\ZZ}}\left\|\sum_{l\geq 1}\sum_{|k|=N_l}^{N_{l+1}}T^k x_{{n_k}}\right\|^{2d}d\bar \mu\big((n_k)\big)\\
&\leq&\sum_{l_1,\dots,l_{2d}\geq 1}\int_{\NN^{\ZZ}}\prod_{i=1}^{2d}\left\|\sum_{|k|=N_{l_i}}^{N_{{l_i}+1}}T^k x_{{n_k}}\right\|d\bar \mu\big((n_k)\big).
\end{eqnarray*}
We then apply H\"older's inequality to get 
$$\int_X \|x\|^{2d}\leq\sum_{l_1,\dots,l_{2d}\geq 1}\prod_{i=1}^{2d}\left(\int_{\NN^{\ZZ}}\left\|\sum_{|k|=N_{l_i}}^{N_{{l_i}+1}}T^k x_{{n_k}}\right\|^{2d}d\bar \mu\big((n_k)\big)\right)^{1/2d}.$$
We fix some $l\geq 1$ and we want to estimate $\int_{\NN^{\ZZ}}\left\|\sum_{|k|=N_l}^{N_{l+1}}T^k x_{{n_k}}\right\|^{2d}d\bar \mu\big((n_k)\big)$.
Let $(n_k)\subset\NN^{\ZZ}$ and let us write
\begin{eqnarray*}
\left\|\sum_{|k|=N_l}^{N_{l+1}}T^k x_{{n_k}}\right\|^{2d}&\leq&2^{2d}\left(\left\|\sum_{\substack{|k|=N_l\\n_k\leq l}}^{N_{l+1}}T^k x_{n_k}\right\|^{2d}+
\left\|\sum_{\substack{|k|=N_l\\n_k> l}}^{N_{l+1}}T^k x_{n_k}\right\|^{2d}\right)\\
&\leq&\frac{2^{2d}}{2^{ld}}+2^{4d-1}(N_{l+1}-N_l)^{2d-1}\sum_{|k|=N_l}^{N_{l+1}}\left\|T^k x_{n_k}\right\|^{2d}\times\mathbf 1_{\{n_k>l\}}.
\end{eqnarray*}
We integrate this inequality over $\NN^{\ZZ}$ to find 
\begin{eqnarray*}
\int_{\NN^{\ZZ}}\left\|\sum_{|k|=N_l}^{N_{l+1}}T^k x_{{n_k}}\right\|^{2d}d\bar \mu\big((n_k)\big)&\leq&\frac{2^{2d}}{2^{ld}}+2^{4d-1}(N_{l+1}-N_l)^{2d-1}\sum_{|k|=N_l}^{N_{l+1}}\sum_{m>l}p_m\left\|T^k x_{m}\right\|^{2d}\\
&\leq&\frac{2^{2d}}{2^{ld}}+2^{4d}(N_{l+1}-N_l)^{2d}\sum_{m>l}p_m\omega_0(m)^{2d}\\
&\leq&\frac{2^{2d}}{2^{ld}}+2^{4d}(N_{l+1}-N_l)^{2d}p_l\max\big(\omega_0(l),\omega_0(2d)\big)^{2d}\\
&\leq&\frac{2^{2d}}{2^{ld}}+2^{4d}(N_{l+1}-N_l)^{2d}p_l^{1/2}\omega_0(2d)^{2d},
\end{eqnarray*}
since we assumed $\omega_0(l)^{2d}p_l\leq p_l^{1/2}.$ We take the power $1/2d$ and we sum the inequalities to get
\begin{eqnarray*}
\int_X \|x\|^{2d}d\mu(x)&\leq&C\left(\sum_{l\geq 1}\left(\frac{2}{2^{l/2}}+4(N_{l+1}-N_l)p_l^{1/4d}\right)\right)^{2d}\omega_0(2d)^{2d}\\
&\leq&C_d.
\end{eqnarray*}
Thus, $\|\cdot\|^d$ belongs to $L^2(X,\mathcal B,\mu)$. Contrary to what happens in Lemma \ref{LEMXBMUISLARGE}, we cannot control its norm by $\omega_0(2d)^d$, but only by some constant $C_d$ which can be much larger. This also affects Lemma \ref{LEMCOVARIANCE3}, where we have to replace $C^d r^d\omega_0(d)^d$ by some constant $C_d$ depending on $d$. However, when we want to study the sequence of covariances and the validity of the Central Limit Theorem for the sequence $(f\circ T^n)$, with $f$ a polynomial, this is unimportant. Indeed, the sum over $r$ which appears in both proofs is now a finite sum.
\end{proof}

\subsection{Central limit theorems and the Gaussian measure}
When $T$ satisfies the assumptions of Theorem \ref{THMMURILLOPERIS}, we have two ways to define a $T$-invariant Gaussian measure on $X$ with full support:
the Gaussian measure of \cite{BAYMATHERGOBEST} and the measure constructed in \cite{MuPe13} or in Section \ref{SECPROOF}. The Gaussian measure is probably
simpler. However, it is easier to understand why $T$ is ergodic with respect to the measure of Section \ref{SECPROOF}: it behaves like a Bernoulli shift.
This was very useful to apply the martingale method, in particular to have a "canonical" choice of the filtration $(\mathcal F_i)$ to apply Theorem A.

If we want to apply Theorem A with the Gaussian ergodic measure, it is not clear which filtration could be convenient. A particular case is that of 
backward shifts: in that case, $T$ is already a shift! Hence, on $\ell^2(\mathbb Z_+)$, we can prove a statement similar to Theorem \ref{THMTCLBS} with a Gaussian measure and replacing $E_\omega$ by some subspace of $L^2$ similar to those of Devinck.

\begin{question}
Let $T\in\mathfrak L(X)$ satisfying the assumptions of Theorem \ref{THMERGOBEST} and let $\mu$ be a $T$-invariant  and ergodic Gaussian measure on $X$
with full support. Does there exist a big subspace $E\subset L^2(\mu)$ such that any $f\in E$ satisfies the CLT?
\end{question}

\subsection{Central limit theorem for linear forms}
It is a little bit deceiving in the statement of  Theorem \ref{THMMAINTCL} that the sequence $(\veps_k)$ is not directly related to the decay needed on $\|T^k x\|$ to obtain a CLT. 
It turns out that this is the case if we restrict ourselves to linear forms.

\begin{theorem}
Let $T\in\LX$, let $\alpha>1$ and let $\omega:\mathbb N\to(1,+\infty)$ going to $+\infty$. Suppose that, for any function $\omega_0:[1,+\infty)\to(1,+\infty)$ going to infinity, one can find a sequence $\mathcal D=(x_n)_{n\geq 1}\subset X$
and a  sequence of maps $S_n:\mathcal D\to X$, $n\geq 0$, such that
\begin{enumerate}[(i)]
\item For any $x\in\mathcal D$, $\sum_{n\geq 0} T^nx$ converges unconditionally;
\item For any $x\in\mathcal D$, $\sum_{n\geq 0} S_n x$ converges unconditionally;
\item For any $x\in\mathcal D$, $T^n S_n x= x$ and $T^mS_n x=S_{n-m}x$ for any $n>m$;
\item For any sequence $(n_k)\subset\NN^{\ZZ}$ such that $\sum_{k\geq 0}T^k x_{n_k}$ and $\sum_{k<0}S_{-k} x_{n_k}$
are convergent, 
$$\left\| \sum_{k\geq 0}T^k x_{n_k}\right\|+\left\|\sum_{k<0}S_{-k} x_{n_k}\right\|\leq \sum_{k\in\ZZ}\omega_0(n_k)k^{-\alpha}.$$
\item $\textrm{span}\big(\{T^k x_n;S_k x_n\};\ k\geq 0,\ n\geq 1\big)$ is dense in $X$.
\end{enumerate}
Then there exists a $T$-invariant strongly mixing Borel probability measure $\mu$ on $X$
with full support such that $E_\omega \subset L^2(X,\mathcal B,\mu)$. Moreover, for any $x^*\in X^*$,  the sequence $\dis\frac{1}{\sqrt n}(x^*+\dots+ x^*\circ T^{n-1})$
converges in distribution to a Gaussian random variable of finite variance.
\end{theorem}
\begin{proof}
Let $x^*\in X^*$ and let us show that $x^*$ satisfies the CLT. 
Let $F=f\circ\phi=\sum_{j\in\mathbb Z,l\in\mathbb N}a_{l,j}e_{l,j}$. A look at the proof of Theorem \ref{THMMAINTCL} shows that we only need to prove 
\begin{eqnarray}
\label{EQCLTLF1}
\sum_{j\geq 0}\sum_{l\geq 1}\left|\sum_{p=0}^{n-1}a_{l,j+p}\right|^2&=&O(n^{1-\veps})\\
\label{EQCLTLF2}
\sum_{j<-n}\sum_{l\geq 1}\left|\sum_{p=0}^{n-1}a_{l,j+p}\right|^2&=&O(n^{1-\veps})
\end{eqnarray}
for some $\veps>0$. Let $j\geq 0$ and $l\geq 1$. Like in Lemma \ref{LEMCOVARIANCE3},
\begin{eqnarray*}
a_{l,j}&=&\int_{\mathbb N}e_{l,j}\big((n_k)\big)\langle x^*,\phi\big((n_k\big))\rangle d\bar\mu\big((n_k)\big)\\
&=&\sum_{m\geq l}p_me_l(m)\langle x^*,T^j x_m\rangle.
\end{eqnarray*}
By linearity of $x^*$, 
$$\sum_{p=0}^{n-1}a_{l,j+p}=\sum_{m\geq l}p_me_l(m)\langle x^*,\sum_{p=0}^{n-1}T^{j+p}x_m\rangle.$$
We can now use Assumption (iv) and argue as in Lemma \ref{LEMCOVARIANCE3}:
\begin{eqnarray*}
\left|\sum_{p=0}^{n-1}a_{l,j+p}\right|&\leq&C\sum_{m\geq l}p_m |e_l(m)|\omega_0(m)\|x^*\|\sum_{p=0}^{n-1}\frac 1{(1+j+p)^\alpha}\\
&\leq&C\sqrt{p_l}\omega_0(l)\|x^*\|\sum_{p=0}^{n-1}\frac 1{(1+j+p)^\alpha}.
\end{eqnarray*}
This implies, arguing as in Lemma \ref{LEMSUMFOURIER},
$$\sum_{j\geq 0}\sum_{l\geq 1}\left|\sum_{p=0}^{n-1}a_{l,j+p}\right|^2\leq C\|x^*\|^2\sum_{j\geq 0}\left|\sum_{p=0}^{n-1}\frac 1{(1+j+p)^\alpha}\right|^2\leq C\|x^*\|^2n^{3-2\alpha}.$$
Since $3-2\alpha<1$, this shows (\ref{EQCLTLF1}). The proof of (\ref{EQCLTLF2}) is completely similar and omitted.
\end{proof}

This theorem is interesting for operators such that $\sum_n T^n x$ converges unconditionnaly whereas $\sum_n \|T^n x\|=+\infty$. For instance, we get the following corollary for backward shifts, where the value $1/p$ is optimal. 
\begin{corollary}
Let $\omega:\mathbb N\to(1,+\infty)$ going to $+\infty$ and let $B_{\mathbf w}$ be a bounded backward weighted shift
on $\ell^p(\mathbb Z_+)$. Suppose that there exists $\alpha>1/p$ such that, for any $n\geq 1$,
$w_1\cdots w_n\geq Cn^{\alpha}.$
Then there exists a  $B_{\mathbf w}$-invariant strongly mixing Borel probability measure $\mu$ on $\ell^p(\mathbb Z_+)$ with full support
such that $E_\omega\subset L^2(X,\mathcal B,\mu)$.
Moreover,  for any $x^*\in \ell^q(\mathbb Z_+)$, $\frac1p+\frac 1q=1$, the sequence $\dis\frac{1}{\sqrt n}(x^*+\dots+ x^*\circ T^{n-1})$
converges in distribution to a Gaussian random variable of finite variance.
\end{corollary}

More surprinzingly, we obtain that a central limit theorem holds in the context of parabolic composition operators. This was unavailable with Theorem \ref{THMMAINTCL}.
\begin{corollary}
Let $\phi$ be a parabolic automorphism of the disk. There exists a $C_\phi$-invariant
strongly mixing Borel probability measure $\mu$ on $H^2(\DD)$ such that $(H^2)^*\subset L^2(\mu)$ and any $x^*$ in $(H^2)^*$ satisfies the CLT.
\end{corollary}
\begin{question}
Do the above corollaries remain true if we consider polynomials instead of linear forms?
\end{question}

\subsection{Fr\'echet spaces}
Theorem \ref{THMMURILLOPERIS} was proved in \cite{MuPe13} for an $F-$space. For convenience, we restrict ourselves to Banach spaces;
however, there are interesting examples beside this context, especially in the Fr\'echet space setting.

\begin{question}
Can we extend Theorem \ref{THMMAINTCL} to Fr\'echet spaces?
\end{question}

\providecommand{\bysame}{\leavevmode\hbox to3em{\hrulefill}\thinspace}
\providecommand{\MR}{\relax\ifhmode\unskip\space\fi MR }
% \MRhref is called by the amsart/book/proc definition of \MR.
\providecommand{\MRhref}[2]{%
  \href{http://www.ams.org/mathscinet-getitem?mr=#1}{#2}
}
\providecommand{\href}[2]{#2}


\begin{thebibliography}{MAP13}

\bibitem[BG06]{BAYGRITAMS}
F.~Bayart and S.~Grivaux, \emph{Frequently hypercyclic operators}, Trans. Amer.
  Math. Soc. \textbf{358} (2006), no.~11, 5083--5117.

\bibitem[BG07]{BAYGRIPLMS}
\bysame, \emph{{Invariant Gaussian measures for operators on Banach spaces and
  linear dynamics}}, Proc. Lond. Math. Soc. \textbf{94} (2007), 181--210.

\bibitem[BM09]{BM09}
F.~Bayart and \'E. Matheron, \emph{Dynamics of linear operators}, Cambridge
  Tracts in Math, vol. 179, Cambridge University Press, 2009.

\bibitem[BM11]{BAYMATHERGOBEST}
\bysame, \emph{Mixing operators and small subsets of the circle}, preprint
  (2011), arXiv:1112.1289.

\bibitem[BR13]{BAYRUZSA}
F.~Bayart and I.~Ruzsa, \emph{{Difference sets and frequently hypercyclic
  weighted shifts}}, preprint (2013).

\bibitem[CB05]{CLB05}
S.~Cantat and S.~Le Borgne, \emph{{Th\'eor\`eme central limite pour les
  endomorphismes holomorphes et les correspondances modulaires}}, Int. Math.
  Res. Not. \textbf{56} (2005), 3479--3510.

\bibitem[Che95]{Che95}
N.I. Chernov, \emph{{Limit theorems and Markov approximations for chaotic
  dynamical systems}}, Probab. Theory Relat. Fields \textbf{101} (1995),
  321--362.

\bibitem[Dev13]{Dev13}
V.~Devinck, \emph{{Strongly mixing operators on Hilbert spaces and speed of
  mixing}}, Proc. London Math. Soc. \textbf{(to appear)} (2013).

\bibitem[DS06]{DiSi06}
T.C Dinh and N.~Sibony, \emph{Decay of correlations and central limit theorem
  for meromorphic maps}, Comm. Pure Appl. Math. \textbf{59} (2006), 754--768.

\bibitem[Dup10]{Dup10}
C.~Dupont, \emph{Bernoulli coding map and almost-sure invariance principle for endomorphisms of $\mathbb P(k)$},
Probab. Theory Related Fields \textbf{146} (2010), 337--359.

\bibitem[Fly95]{Fly95}
E.~Flytzanis, \emph{{Unimodular eigenvalues and linear chaos in Hilbert
  spaces}}, Geom. Funct. Anal. \textbf{5} (1995), 1--13. \MR{95k:28034}

\bibitem[GEP11]{GePeBook}
K.-G. Grosse-Erdmann and A.~Peris, \emph{Linear chaos}, Springer, 2011.

\bibitem[Gor69]{Gor69}
M.I. Gordin, \emph{{The Central Limit Theorem for stationary processes}},
  Soviet. Math. Dokl. \textbf{10} (1969), 1174--1176.

\bibitem[LB99]{LB99}
S.~Le~Borgne, \emph{Limit theorems for non-hyperbolic automorphisms of the
  torus}, Israel J. Math. \textbf{109} (1999), 61--73.

\bibitem[Liv96]{Liv96}
C.~Liverani, \emph{Central limit theorem for deterministic systems}, Pitman
  Res. Notes Math. Ser., vol. 362, 1996, pp.~56--75.

\bibitem[MAP13]{MuPe13}
M.~Murillo-Arcila and A.~Peris, \emph{Strong mixing measures for linear
  operators and frequent hypercyclicity}, J. Math. Anal. Appl. \textbf{398}
  (2013), 462--465.

\bibitem[MW00]{MW00}
M.~Maxwell and M.~Woodroofe, \emph{{Central limit theorems for additive
  functionals of Markov chains}}, Ann. Probab. \textbf{28} (2000), 713--724.

\bibitem[Vol90]{Vol90}
D.~Voln\'y, \emph{On limit theorems and category for dynamical systems},
  Yokohama Math. J. \textbf{38} (1990), 29--35.

\bibitem[Vol06]{Vol06}
\bysame, \emph{{Martingale approximation of non adapted stochastic processes
  with nonlinear growth of variance}}, Dependence in Probability and
  Statistics, Lecture Notes in Statistics, vol. 187, Springer-Verlag, 2006,
  pp.~141--156.

\end{thebibliography}
\end{document}